\documentclass[a4paper,11pt,DIV=11,% decrease borders by 2 (std 10 for 11pt)
abstract=on% return to old style centered abstract
]{scrartcl}
\usepackage{mathtools}
\mathtoolsset{showonlyrefs}
\usepackage{amsmath, amsthm, amssymb}
\usepackage{fontenc}
\usepackage[utf8]{inputenc}
\usepackage{graphicx}
\usepackage{xargs}
\usepackage[subrefformat=parens,labelformat=parens]{subcaption}
\usepackage[section]{algorithm}
\usepackage[noend]{algpseudocode}
\usepackage{enumerate}
\usepackage[shortlabels]{enumitem}
\usepackage{listings,microtype,tikz,pgfplots}
\usepackage{geometry}
\usepackage{hyperref}
\usepackage{environ,ifthen}
\usepackage{multirow}
\usepackage{verbatim}
\usepackage{mathrsfs}
\usepackage{bbold}
\usepackage{accents}

\newtheorem{Theorem}{Theorem}[section]
\newtheorem{Proposition}[Theorem]{Proposition}
\newtheorem{Lemma}[Theorem]{Lemma}

\newtheorem{Assumption}[Theorem]{Assumption}
\theoremstyle{definition}
\newtheorem{Definition}[Theorem]{Definition}
\newtheorem{Remark}[Theorem]{Remark}

\setkomafont{sectioning}{\rmfamily\bfseries}
\setkomafont{title}{\rmfamily}
\setkomafont{descriptionlabel}{\rmfamily\bfseries}

\usetikzlibrary{spy}
\usetikzlibrary{calc}
\usepackage{algorithm}
\usepackage{booktabs}

\usepackage[noend]{algpseudocode}

\usepackage{setspace}			% 1,5-facher Zeilenabstand
\usepackage{flafter}			% Floats werden nicht vor ihrer Definitionsstelle gesetzt

\usepackage{appendix}
\usepackage{hyperref}	
\hypersetup{colorlinks=true,linkcolor=black,urlcolor=blue,citecolor=black}		% für Hyperlinks in PDF
\usepackage{placeins} 		%	 für FloatBarrier, das die Ausgabe der Grafiken bis zu diesem Punkt erzwingt

\usepackage{listings}
\usepackage{empheq}

\newcommand{\N}{\mathbb{N}}

\newcommand{\R}{\mathbb{R}}
\newcommand{\C}{\mathbb{C}}

\newcommand{\SP}{\mathbb{S}}

\newcommand{\NN}{\mathcal{N}}

\newcommand{\E}{\mathbb{E}}

\newcommand{\abs}[1]{\left| #1 \right|}

\newcommand{\dx}[1][x]{\,\mathrm{d}#1}

\newcommand{\tT}{\mathrm{T}}

\newcommand{\e}{\mathrm{e}}

\newcommand{\prox}{\mathrm{prox}}
\newcommand{\dom}{\mathrm{dom}\,}
\newcommand{\dist}{\,\mathrm{dist}}
\newcommand{\St}{\,\mathrm{St}}
\def\tT{{\mbox{\tiny{T}}}}

\newcommand{\ubar}[1]{\underaccent{\bar}{#1}}

\renewcommand\theta{\vartheta}

\DeclareMathOperator*{\argmin}{argmin}

\DeclareMathOperator{\Cov}{Cov}

\DeclareMathOperator{\SPD}{SPD}

\DeclareMathOperator{\Sym}{Sym}

\allowdisplaybreaks[4]

\begin{document}
\title{Inertial Stochastic PALM (iSPALM) and Applications in Machine Learning}
\author{
Johannes Hertrich\footnotemark[1]
\and
Gabriele Steidl\footnotemark[1]
}

\maketitle

\footnotetext[1]{Institute of Mathematics,
TU Berlin,
Straße des 17. Juni 136, 
 D-10623 Berlin, Germany,
\{j.hertrich, steidl\}@math.tu-berlin.de.
} 

\begin{abstract}
Inertial algorithms for minimizing nonsmooth and nonconvex functions
as the inertial proximal alternating linearized minimization algorithm (iPALM)
have demonstrated their superiority with respect to computation time over their non inertial variants.
In many  problems in imaging and machine learning, 
the objective functions have a special form involving huge data which encourage the
application of stochastic algorithms. While algorithms based on stochastic gradient descent are still used in the
majority of applications, recently also stochastic algorithms for minimizing nonsmooth and nonconvex functions
were proposed.

In this paper, we derive an inertial variant of a stochastic PALM algorithm 
with variance-reduced gradient estimator, called iSPALM, 
and prove linear convergence of the algorithm under certain assumptions.
Our inertial approach can be seen as generalization of momentum methods widely used to speed up and stabilize optimization algorithms, in particular in machine learning,
to nonsmooth problems. 
Numerical experiments for learning the weights of a so-called
proximal neural network and the parameters of Student-$t$ mixture models 
show that our new algorithm outperforms both 
stochastic PALM and 
its deterministic counterparts.
\end{abstract}

\section{Introduction}

Recently, duality concepts were successfully applied for minimizing
nonsmooth and nonconvex functions appearing in certain applications in image and data processing.
A frequently applied algorithm in this direction is the proximal alternating linearized minimization algorithm (PALM)
by Bolte, Teboulle and Sabach \cite{BST2014} based on results in \cite{AB2009,ABRS2010}.
Pock and Sabach \cite{PS2016} realized that the convergence speed of PALM can be considerably improved
by inserting some nonexpensive inertial steps and called the accelerated algorithm iPALM.
In many  problems in imaging and machine learning, 
parts of the objective function
can be often written as sum of a huge number of functions sharing the same structure.
In general the computation of the gradient of these parts is too time and storage consuming
so that stochastic gradient approximations were applied, see, e.g. \cite{Bottou2010} and the references therein. 
A combination of the simple stochastic gradient descent (SGD) estimator 
with PALM was first discussed by  Xu and Yin in \cite{XY2015}. 
The authors refer to their method as
block stochastic gradient iteration and do not mention the connection to PALM.
Under rather hard assumptions on the objective function $F$, 
they proved that the sequence $(x^k)_k$ produced by their algorithm is such that
$\E\left(\dist(0,\partial F(x^k) \right)$ converges to zero as $k \to \infty$.
Another idea for a stochastic variant of PALM was proposed by Davis et al. \cite{DEU2016}.
The authors introduce an asynchronous variant of PALM with stochastic noise in the gradient and called it SAPALM.
Assuming an explicit bound of the variance of the noise, they proved certain convergence results.
Their approach requires an explicit bound on the noise, which is not fulfilled for the gradient estimators
considered in this paper.
Further, we like to mention that a stochastic variant of the primal-dual algorithm of Chambolle and Pock \cite{CP11}
for solving convex problems was developed in \cite{CERS2018}.

Replacing the simple stochastic gradient descent estimators by more sophisticated so-called variance-reduced
gradient estimators, Driggs et al. \cite{DTLDS2020} 
could weaken the assumptions on the objective function in \cite{XY2015}
and improve the estimates on the convergence rate of a stochastic PALM algorithm.
They called the corresponding algorithm SPRING.
Note that the advantages of variance reduction to accelerate stochastic gradient methods 
were discussed by several authors, see, e.g. \cite{JZ2013,RHSPS2016}.

In this paper, we merge a stochastic PALM algorithm 
with an inertial procedure 
to obtain a new iSPALM algorithm.
The inertial parameters can also be viewed as a generalization of momentum parameters 
to nonsmooth problems. 
Momentum parameters are widely used to speed up and stabilize optimization algorithms based on (stochastic)
gradient descent.
In particular, for machine learning applications it is known that momentum algorithms 
~\cite{N1983,P1964,Q1999,RHW1986} as well as their stochastic modifications like the Adam optimizer \cite{KB2014} 
perform much better than a plain (stochastic) gradient descent, see e.g. \cite{GLZX2019, SMDH2013}.
From this point of view, inertial or momentum parameters are one of the core ingredients for an efficient optimization
algorithm to minimize the loss in data driven approaches.
We examine the convergence behavior of iSPALM both theoretically and numerically.
Under certain assumptions on the parameters of the algorithm which also appear in the iPALM algorithm, we show
that iSPALM converges linearly. In particular, we have to adapt the definition of 
variance-reduced gradient estimators to the sequence produced by iSPALM.
More precisely, we have to introduce inertial variance-reduced gradient estimators.
In the numerical part, we focus on two examples, namely (i) MNIST classification with 
proximal neural networks (PNNs), and (ii) parameter learning for Student-$t$ mixture models (MMs).

PNNs basically replace the standard layer $\sigma(Tx+b)$ of a feed-forward neural network by 
$T^\tT\sigma(Tx+b)$ and require that $T$ is an element of the (compact) Stiefel manifold, 
i.e.\ has orthonormal columns, see \cite{HHNPSS2020,HNS2020}.
This implies that PNNs are $1$-Lipschitz 
and hence more stable under adversarial attacks than a neural network of comparable size without the
orthogonality constraints.
While the PNNs were trained in \cite{HHNPSS2020} using a SGD on the Stiefel manifold, 
we train it in this paper by adding the characteristic function of the feasible weights to the loss for incorporating the orthogonality constraints and use PALM, iPALM, SPRING and iSPALM for the optimization.

Learned MMs provide a powerful tool in data and image processing.
While Gaussian MMs are mostly used in the field, more robust methods
can be achieved by using heavier tailed distributions, as, e.g.\ the Student-$t$ distribution.
In~\cite{VS14}, it was shown that Student-$t$ MMs are superior 
to Gaussian ones for mode\-ling image patches and the authors proposed an application in image compression. 
Image denoising based on Student-$t$ models was addressed in~\cite{LS2019} and image deblurring in~\cite{DHWMZ2019,YYG2018}.
Further applications include robust image segmentation~\cite{BM18,NW12,SNG07} and superresolution \cite{Hertrich2020}
as well as registration~\cite{GNL09,ZZDZC14}. For learning MMs a maximizer of the corresponding log-likelihood
has to be computed. Usually an expectation maximization (EM) algorithm \cite{LLT89,McLK1997,PM00} 
or certain of its acceleration \cite{Byrne2017,MVD97,vanDyk1995} are applied for this purpose.
However, if the MM has many components and we are given large data,  a stochastic optimization approach appears to be more efficient.
Indeed, recently, also stochastic variants of the EM algorithm were proposed \cite{CM2009,TZTZ2018}, 
but show various disadvantages
and we are not aware of a circumvent convergence result for these algorithms. 
In particular, one assumption on the stochastic EM algorithm is that the underlying distribution family is an exponential family, which is not the case for MMs.
In this paper, 
we propose for the first time to use the (inertial) PALM algorithms as well as their stochastic variants 
for maximizing a modified version of the log-likelihood function. 
\\

This paper is organized as follows:
In Section \ref{sec:prelim}, we provide the notation used throughout the paper.
To understand the differences of existing algorithms to our novel one, 
we discuss PALM and iPALM together with convergence results in 
Section \ref{sec:PALM}.
Section \ref{sec:sPALM} contains their stochastic variants, where where our iSPALM is
new.
We discuss the convergence behavior of iSPALM in Section \ref{sec:conv}.
In Section \ref{sec:student-t}, we propose a model for learning the parameters of Student-$t$ MMs
based on its log-likelihood function. We show how (inertial) PALM and its stochastic variants (inertial) SPRING
can be used for optimization.
Further, we prove that our model fulfills the assumptions on the convergence of these algorithms.
Section \ref{sec:numerics}
compares the performance of the four algorithms for two examples. We provide the code online\footnote{\url{https://github.com/johertrich/Inertial-Stochastic-PALM}}.
Finally, conclusions are drawn and directions of further research are addressed in Section \ref{sec:concl}.

%---------------------------------------
\section{Preliminaries} \label{sec:prelim}
%---------------------------------------
In this section,  we introduce the basic notation and results which we will use throughout this paper. 

For an proper and lower semi-continuous function $f\colon\R^d\to  (-\infty,\infty]$ 
and $\tau>0$ the \emph{proximal mapping} $\prox_\tau^f\colon\R^d\to\mathcal{P}(\R^d)$ is defined by
$$
\prox_\tau^f(x) \coloneqq\argmin_{y\in\R^d} \left\{ {\tfrac{\tau}2\|x-y\|^2+f(y)} \right\},
$$
where $\mathcal{P}(\R^d)$ denotes the power set of $\R^d$.
The proximal mapping admits the following properties, see e.g.\ \cite{RW98}.

\begin{Proposition}
Let $f\colon\R^d\to\R$ be proper and lower semi-continuous with $\inf_{\R^d}f>-\infty$. Then, the following holds true.
\begin{enumerate}
\item The set $\prox_\tau^f(x)$ is nonempty and compact for any $x\in\R^d$ and $\tau>0$.
\item If $f$ is convex, then $\prox_\tau^f(x)$ contains exactly one value for any $x\in\R^d$ and $\tau>0$.
\end{enumerate}
\end{Proposition}

To describe critical points, we will need the definition of (general) subgradients.

\begin{Definition}
Let $f\colon\R^d\to  (-\infty,\infty]$ be a proper and lower semi-continuous function and $v\in\R^d$.
Then we call
\begin{enumerate}
\item $v$  a \emph{regular subgradient} of $f$ at $\bar x$, written $v\in\hat\partial f(\bar x)$, if for all $x \in \R^d$,
\begin{align}
f(x)\geq f(\bar x)+\langle v,x-\bar x\rangle+o(\|x-\bar x\|).
\end{align}
\item $v$ a (general) \emph{subgradient} of $f$ at $\bar x$, written $v\in\partial f(\bar x)$, if there are sequences $x^k \to \bar x$ 
and $v^k\in \hat\partial f(x^k)$ with $v^k\to v$ as $k\to\infty$.
\end{enumerate}
\end{Definition}

The following proposition lists useful properties of subgradients.

\begin{Proposition}[Properties of Subgradients] \label{prop:subgrad}
Let $f\colon\R^{d_1}\to(-\infty,\infty]$ and $g\colon\R^{d_2}\to(-\infty,\infty]$ be proper 
and lower semicontinuous and let $h\colon\R^{d_1}\to\R$ be continuously differentiable. 
Then the following holds true.
\begin{enumerate}
\item For any $x\in\R^{d_1}$, we have $\hat\partial f(x)\subseteq \partial f(x)$. 
If $f$ is additionally convex, we have $\hat\partial f(x)=\partial f(x)$.
\item For $x\in\R^{d_1}$ with $f(x)<\infty$, it holds 
$$
\hat\partial (f+h)(x)=\hat\partial f(x)+\nabla h(x) \quad\text{and}\quad \partial (f+h)(x)=\partial f(x)+\nabla h(x).
$$
\item If $\sigma(x_1,x_2)=f_1(x_1)+f_2(x_2)$, then
$$
\left(\begin{array}{c}\hat\partial_{x_1} f_1(\bar x_1)\\
\hat\partial_{x_2} f_2(\bar x_2)\end{array}\right)
\subseteq\hat\partial \sigma(\bar x_1,\bar x_2)
\quad\text{and}\quad
\left(\begin{array}{c}\partial_{x_1} f_1(\bar x_1)\\
\partial_{x_2} f_2(\bar x_2)\end{array}\right)
\subseteq\partial \sigma(\bar x_1,\bar x_2).
$$
\end{enumerate}
\end{Proposition}

\begin{proof}
Part (i) was proved in \cite[Theorem 8.6 and Proposition 8.12]{RW98} and part (ii) in \cite[Exercise 8.8]{RW98}. 
Concerning part (iii) we have for $v_{x_i} \in \hat\partial_{x_i} f(\bar x_i)$, $i=1,2$
that for all $(x_1,x_2) \in \R^d \times \R^d$ it holds
$$
\sigma(x_1,x_2)=f_1(x_1)+f_2(x_2)
\geq 
\sum_{i=1}^2
f_i(\bar x_i)+\langle v_{x_i},x_i-\bar x_i \rangle + o(\|x_i-\bar x_i\|).
$$
This proves the claim for regular subgradients. 

For general subgradients consider $v_{x_i}\in\partial_{x_i} f_i(\bar x_i)$, $i=1,2$
By definition there exist sequences $x_i^k\to \bar x_i$ and $v_{x_i}^k\to v_{x_i}$ with $v_{x_i}^k \in\hat\partial_{x_i} f_i(x_i^k)$, $i=1,2$. 
By the statement for regular subgradients we know that $(v_{x_1}^k,v_{x_2}^k)\in\hat\partial\sigma(x_1^k,x_2^k)$. 
Thus, it follows by definition of the general subgradient that $(v_{x_1},v_{x_2})\in\partial \sigma (\bar x_1,\bar x_2)$.
\end{proof}

We call $(x_1,x_2)\in\R^{d_1}\times\R^{d_2}$ a \emph{critical point} of 
$F$ if $0\in\partial F(x_1,x_2)$.
By \cite[Theorem 10.1]{RW98} we have that any local minimizer $\hat x$ 
of a proper and lower semi-continuous function $f\colon\R^d\to(-\infty,\infty]$ fulfills 
$$0\in\hat\partial f(\hat x)\subseteq \partial f(\hat x).$$
In particular, it is a critical point of $f$. 
Further, we have by  Proposition \ref{prop:subgrad} that
$
\hat x \in \prox_\tau^f(x)
$
implies 
\begin{equation} \label{star}
0\in\tau(\hat x - x)+\hat\partial f(y)\subseteq \tau(\hat x - x)+\partial f(y).
\end{equation}

In this paper, we consider functions $F\colon \R^{d_1}\times \R^{d_2}\to(-\infty,\infty]$ of the form
\begin{align} \label{eq:PALM_min_general}
F(x_1,x_2)=H(x_1,x_2)+f(x_1)+g(x_2)
\end{align}
with  proper, lower semicontinuous  functions
$f\colon \R^{d_1}\to(-\infty,\infty]$ and $g\colon\R^{d_2} \to (-\infty,\infty]$
bounded from below 
and a continuously differentiable function $H\colon \R^{d_1}\times\R^{d_2}\to\R$.
Further, we assume throughout this paper that
\begin{align} \label{eq:PALM_min_general_1}
\ubar F\coloneqq\inf_{(x_1,x_2) \in \R^{d_1}\times\R^{d_2}} F (x_1,x_2) >-\infty.
\end{align}
By Proposition \ref{prop:subgrad} it holds
\begin{align}
\left(\begin{array}{c}\partial_{x_1}F(x_1,x_2)\\\partial_{x_2}F(x_1,x_2)\end{array}\right)=&\nabla H(x_1,x_2)
+\left(\begin{array}{c}\partial_{x_1}f(x_1)\\
\partial_{x_2}g(x_2)\end{array}\right)\\
\subseteq& \nabla H(x_1,x_2) +\partial (f+g)(x_1,x_2)=\partial F(x_1,x_2). \label{subdiff_calc}
\end{align}

The \emph{generalized gradient}  of $F\colon \R^{d_1}\times \R^{d_2}\to(-\infty,\infty]$ 
was defined in \cite{DTLDS2020} as set-valued function
$$
\mathscr{G}F_{\tau_1,\tau_2}(x_1,x_2)
\coloneqq
\left(
\begin{array}{c}\tau_1(x_1-\prox^f_{\tau_1}(x_1-\tfrac1{\tau_1}\nabla_{x_1}H(x_1,x_2)))\\
\tau_2(x_2-\prox^g_{\tau_2}(x_2-\tfrac1{\tau_2}\nabla_{x_2}H(x_1,x_2)))\end{array}\right).
$$
To motivate this definition, note that $0\in \mathscr{G}F_{\tau_1,\tau_2}(x_1,x_2)$ 
is a sufficient criterion for $(x_1,x_2)$ being a critical point of $F$. 
This can be seen as follows: 
For $(x_1,x_2)\in\mathscr{G}F_{\tau_1,\tau_2}(x_1,x_2)$ we have
$$
x_1\in\prox_{\tau_1}^f(x_1-\tfrac1{\tau_1}\nabla_{x_1}H(x_1,x_2)).
$$
Using \eqref{star}, this implies 
$$
0\in \tau_1(x_1-x_1+\tfrac1{\tau_1}\nabla_{x_1}H(x_1,x_2))+\partial f(x_1)=\nabla_{x_1}H(x_1,x_2)+\partial f(x_1).
$$
Similarly we get 
$
0\in \nabla_{x_2}H(x_1,x_2)+\partial g(x_2)
$.
By \eqref{subdiff_calc} we conclude that $(x_1,x_2)$ is a critical point of $F$.

%---------------------------------------------------
\section{PALM and iPALM} \label{sec:PALM}
%---------------------------------------------------
In this section, we review PALM \cite{BST2014}  and its inertial version iPALM \cite{PS2016}.

%---------------------------------------------------
\subsection{PALM}
%---------------------------------------------------
The following Algorithm \ref{alg:PALM_general} for minimizing \eqref{eq:PALM_min_general} was proposed in \cite{BST2014}.\\

\begin{algorithm}[!ht]
\caption{Proximal Alternating Linearized Minimization (PALM)}\label{alg:PALM_general}
\begin{algorithmic}
\State Input: $(x_1^{0},x_2^{0})\in\R^{d_1}\times \R^{d_2}$, parameters $\tau_1^k,\tau_2^k$ for $k \in \N_0$.
\For {$k=0,1,...$}
\State Set
$$
x_1^{k+1}\in\prox_{\tau_1^k}^f\big(x_1^{k}-\tfrac1{\tau_1^k}\nabla_{x_1} H(x_1^{k},x_2^{k})\big)
$$
\State Set
$$
x_2^{k+1}\in\prox_{\tau_2^k}^g\big(x_2^{k}-\tfrac1{\tau_2^k}\nabla_{x_2} H(x_1^{k+1},x_2^{k})\big)
$$
\EndFor
\end{algorithmic}
\end{algorithm}

To prove convergence of PALM the following additional assumptions on $H$ are needed:

\begin{Assumption}[Assumptions on $H$]\label{ass:PALM2}
\begin{enumerate}[(i)]
\item
For any $x_1\in\R^{d_1}$, the function $\nabla_{x_2} H(x_1,\cdot)$ is globally Lipschitz continuous with Lipschitz constant $L_2(x_1)$. 
Similarly, for any $x_2\in\R^{d_2}$, 
the function $\nabla_{x_2} H(\cdot,x_2)$ is globally Lipschitz continuous with Lipschitz constant $L_1(x_2)$.
\item
There exist $\lambda_1^-,\lambda_2^-,\lambda_1^+,\lambda_2^+>0$ such that
\begin{align}
\inf\{L_1(x_2^{k}):k\in\N\}\geq\lambda_1^-\quad&\text{and}\quad\inf\{L_2(x_1^{k}):k\in\N\}\geq\lambda_2^-,\\
\sup\{L_1(x_2^{k}):k\in\N\}\leq\lambda_1^+\quad&\text{and}\quad\sup\{L_2(x_1^{k}):k\in\N\}\leq\lambda_2^+.
\end{align}
\end{enumerate}
\end{Assumption}

\begin{Remark}\label{rem:PALM_other_ass} 
Assume that $H\in\C^2(\R^{d_1\times d_2})$ fulfills assumption \ref{ass:PALM2}(i).
Then, the authors of \cite{BST2014} showed, 
that there are partial Lipschitz constants $L_1(x_2)$ and $L_2(x_1)$, such that Assumption \ref{ass:PALM2}(ii) is satisfied. \hfill $\Box$
\end{Remark}

The following theorem was proven in \cite[Lemma 3, Theorem 1]{BST2014}. 
For the definition of KL functions see Appendix \ref{sec:KL}. 
Here we just mention that semi-algebraic functions are KL functions, see, e.g. \cite{BST2014}.

\begin{Theorem}[Convergence of PALM] \label{thm:PALM_convergence}
Let $F\colon \R^{d_1}\times \R^{d_2}\to(-\infty,\infty]$ by given by \eqref{eq:PALM_min_general}.
fulfills the Assumptions \ref{ass:PALM2} and that $\nabla H$ is Lipschitz continuous on bounded subsets of $\R^{d_1}\times\R^{d_2}$.
Let $(x_1^{k},x_2^{k})_k$ be the sequence generated by PALM, 
where the step size parameters fulfill 
$$\tau_1^k  \ge  \gamma_1 L_1(x_2^k), \quad \tau_2^k \ge  \gamma_2 L_2(x_1^{k+1})$$
for some $\gamma_1,\gamma_2 >1$. 
Then, for $\eta \coloneqq \min\{(\gamma_1-1)\lambda_1^-,(\gamma_2-1)\lambda_2^-\}$,
 the sequence $(F(x_1^{k},x_2^{k}))_k$ is nonincreasing and 
$$
\tfrac{\eta}{2}\big\|(x_1^{k+1},x_2^{k+1})-(x_1^{k},x_2^{k})\big\|_2^2 \leq F(x_1^{k},x_2^{k})-F(x_1^{k+1},x_2^{k+1}).
$$
If $F$ is in addition a KL function and the sequence $(x_1^k,x_2^k)_k$ is bounded, then it converges to a critical point of $F$.
\end{Theorem}

%------------------------------------------------------------------------------
\subsection{iPALM}
%------------------------------------------------------------------------------
To speed up the performance of PALM the inertial variant iPALM in Algorithm \ref{alg:iPALM_general} was suggested in \cite{PS2016}.

\begin{algorithm}[!ht]
\caption{Inertial Proximal Alternating Linearized Minimization (iPALM)}\label{alg:iPALM_general}
\begin{algorithmic}
\State Input: $(x_1^{-1},x_2^{-1}) =(x_1^{0},x_2^{0})\in\R^{d_1}\times \R^{d_2}$, 
parameters $\alpha_1^{k},\alpha_2^{k},\beta_1^{k},\beta_2^{k},\tau_1^{k},\tau_2^{k}$ for $k\in \N_0$.
\For {$k= 0,1,...$}
\State Set
\begin{align}
y_1^{k}&=x_1^{k}+\alpha_1^{k}(x_1^{k}-x_1^{k-1})\\
z_1^{k}&=x_1^{k}+\beta_1^{k}(x_1^{k}-x_1^{k-1})\\
x_1^{k+1}&\in\prox_{\tau_1^{k}}^f\big(y_1^{k}-\tfrac1{\tau_1^{k}}\nabla_{x_1} H(z_1^{k},x_2^{k})\big)
\end{align}
\State Set
\begin{align}
y_2^{k}&=x_2^{k}+\alpha_2^{k}(x_2^{k}-x_2^{k-1})\\
z_2^{k}&=x_2^{k}+\beta_2^{k}(x_2^{k}-x_2^{k-1})\\
x_2^{k+1}&\in\prox_{\tau_2^{k}}^g\big(y_2^{k}-\tfrac1{\tau_2^{k}}\nabla_{x_2} H(x_1^{k+1},z_2^{k})\big)
\end{align}
\EndFor
\end{algorithmic}
\end{algorithm}

\begin{Remark}[Relation to Momentum Methods]\label{rem_ipalm_momentum}
The inertial parameters in iPALM can be viewed as a generalization of momentum parameters for nonsmooth functions. 
To see this, note that iPALM with one block, $f=0$ and $\beta^k=0$ reads as
\begin{align}
y^k&=x^k+\alpha^k(x^k-x^{k-1}),\\
x^{k+1}&=y^k-\tfrac1{\tau^k}\nabla H(x^k).
\end{align}
By introducing $g^k\coloneqq x^k-x^{k-1}$, this can be rewritten as
\begin{align}
g^{k+1}&=\alpha^k g^k - \tfrac1{\tau^k} \nabla H(x^k),\\
x^{k+1}&=x^k+g^{k+1}.
\end{align}
This is exactly the momentum method as introduced by Polyak in \cite{P1964}.
Similar, if $f=0$ and $\alpha^k=\beta^k\neq 0$, iPALM can be rewritten as
\begin{align}
g^{k+1}&=\alpha^k g^k - \tfrac1{\tau^k} \nabla H(x^k+\alpha^k g^k),\\
x^{k+1}&=x^k+g^{k+1},
\end{align}
which is known as Nesterov's Accelerated Gradient (NAG) \cite{N1983}.
Consequently, iPALM can be viewed as a generalization of both the classical momentum method and NAG to the nonsmooth case.
Even if there exists no proof of tighter convergence rates for iPALM than for PALM, this motivates that the inertial steps really accelerate PALM, since NAG has tighter convergence rates than a plain gradient descent algorithm.
\hfill $\Box$
\end{Remark}

To prove the convergence of iPALM  the parameters of the algorithm must be carefully chosen.

\begin{Assumption}[Conditions on the Parameters of iPALM]\label{ass:iPALM3}
Let $\lambda_i^+$, $i=1,2$ and $L_1(x_2^k)$, $L_2(x_1^k)$ be defined by Assumption \ref{ass:PALM2}.
There exists some $\epsilon>0$ such that for all $k\in\N$ and $i=1,2$ the following holds true:
\begin{enumerate}[(i)]
\item There exist $0<\bar\alpha_i<\tfrac{1-\epsilon}{2}$ such that $0\leq\alpha_i^{k}\leq\bar\alpha_i$
and $0< \bar\beta_i \le 1$ such that $0\leq\beta_i^{k}\leq\bar\beta_i$.
\item The parameters $\tau_1^{k}$ and $\tau_2^{k}$ are given by
$$
\tau_1^{k}\coloneqq \frac{(1+\epsilon)\delta_1+(1+\bar\beta_1)L_1(x_2^{k})}{1-\alpha_1^{k}} 
\quad\text{and}\quad
\tau_2^{k}\coloneqq \frac{(1+\epsilon)\delta_2+(1+\bar\beta_2)L_2(x_1^ {k+1})}{1-\alpha_2^{k}},
$$
and for $i=1,2$,
$$
\delta_i\coloneqq \frac{\bar \alpha_i+\bar\beta_i}{1-\epsilon-2\bar\alpha_i}\lambda_i^+ .
$$
\end{enumerate}
\end{Assumption}

The following theorem was proven in \cite[Theorem 4.1]{PS2016}.

\begin{Theorem}[Convergence of iPALM]\label{thm:iPALM_convergence}
Let $F\colon \R^{d_1}\times \R^{d_2}\to(-\infty,\infty]$  given by \eqref{eq:PALM_min_general} be a KL function.
Suppose that $H$ 
fulfills the Assumptions \ref{ass:PALM2} and that $\nabla H$ is Lipschitz continuous on bounded subsets of $\R^{d_1}\times\R^{d_2}$. 
Further, let the parameters of iPALM fulfill the parameter conditions \ref{ass:iPALM3}.
If the sequence $(x_1^{k},x_2^{k})_k$ generated by iPALM is bounded, then
it converges to a critical point of $F$.
\end{Theorem}

\begin{Remark}
Even though we cited PALM and iPALM just for two blocks $(x_1,x_2)$ of variables, 
the convergence proofs from \cite{BST2014} and \cite{PS2016} even work with more than two blocks.
\hfill $\Box$
\end{Remark}

%---------------------------------------------------------------
\section{Stochastic Variants of PALM and iPALM} \label{sec:sPALM}
%---------------------------------------------------------------
For many problems in imaging and machine learning
the function $H:\R^{d_1}\times \R^{d_2}\to \R$ in 
\eqref{eq:PALM_min_general} is of the form
\begin{equation} \label{special_form_h}
H(x_1,x_2)=\frac1n\sum_{i=1}^n h_i(x_1,x_2),
\end{equation}
where $n$ is large. Then the computation of the gradients in PALM and iPALM is very time consuming. 
Therefore stochastic approximations of the gradients were considered in the literature.

%-----------------------------------------------
\subsection{Stochastic PALM and SPRING}
%-----------------------------------------------
The idea to combine stochastic gradient estimators with a PALM scheme 
was first discussed by  Xu and Yin in \cite{XY2015}. 
The authors replaced the gradient in Algorithm \ref{alg:PALM_general} by 
the \emph{stochastic gradient descent} (SGD) \emph{estimator}
\begin{equation} \label{special_form}
\tilde\nabla_{x_i} H(x_1,x_2) \coloneqq \frac1b\sum_{j\in B}\nabla_{x_i}h_j(x_1,x_2),
\end{equation}
where $B\subset \{1,...,n\}$ is a random subset (mini-batch) of fixed batch size $b=|I|$. 
This gives Algorithm \ref{alg:SPALM_general} which we call SPALM.

\begin{algorithm}[!ht]
\caption{Stochastic Proximal Alternating Linearized Minimization (SPALM/SPRING)}\label{alg:SPALM_general}
\begin{algorithmic}
\State Input: $(x_1^{0},x_2^{0})\in\R^{d_1}\times \R^{d_2}$, parameters $\tau_1^k,\tau_2^k$  for $k \in \N_0$.
\For {$k=0,1,...$}
\State Set
$$
x_1^{k+1}\in\prox_{\tau_1^k}^f\big(x_1^{k}-\tfrac1{\tau_1^k}\tilde\nabla_{x_1} H(x_1^{k},x_2^{k})\big)
$$
\State Set
$$
x_2^{k+1}\in\prox_{\tau_2^k}^g\big(x_2^{k}-\tfrac1{\tau_2^k}\tilde\nabla_{x_2} H(x_1^{k+1},x_2^{k})\big)
$$
\EndFor
\end{algorithmic}
\end{algorithm}

Xu and Yin showed under rather strong assumptions, 
in particular $f,g$ have to be Lipschitz continuous and the variance of the SGD estimator has to be bounded,
that there exists a subsequence $(x_1^k,x_2^k)_k$ of iterates generated by 
Algorithm \ref{alg:SPALM_general} such that the sequence
$\E\left(\dist(0,\partial F(x_1^k,x_2^k) \right)$ converges to zero as $k \to \infty$.
If $F$, $f$ and $g$ are strongly convex, the authors proved also convergence of the function values to the infimum of $F$.

Driggs et al.  \cite{DTLDS2020} could weaken the assumptions 
and improve the convergence rate by replacing the SGD estimator 
by so-called variance-reduced gradient estimators.
Let $\E_k=\E(\cdot|(x_1^1,x_2^1),(x_1^2,x_2^2),...,(x_1^k,x_2^k))$ 
be the \emph{conditional expectation on the first $k$ sequence elements}.
Then these estimators have to fulfill the following properties.

\begin{Definition}[Variance-Reduced Gradient Estimator]\label{def_var-red}
A gradient estimator 
$\tilde\nabla$ is called \emph{variance-reduced} for a differentiable function $H: \R^{d_1} \times \R^{d_2} \to \R$ 
with constants $V_1,V_2,V_\Upsilon\geq0$ and $\rho\in(0,1]$, 
if for any sequence $(x^k)_k=(x_1^k,x_2^k)_k$ the following holds true:
\begin{enumerate}[(i)]
\item There exist random vectors $v_k^i$, $k\in \N_0$, $i=1,...,s$ 
such that for $\Upsilon_k=\sum_{i=1}^s\|v_k^i\|^2$,
\begin{align}
&\E_k(\|\tilde\nabla_{x_1}H(x_1^k,x_2^k)-\nabla_{x_1}H(x_1^k,x_2^k)\|^2
+\|\tilde\nabla_{x_2}H(x_1^{k+1},x_2^k)-\nabla_{x_2}H(x_1^{k+1},x_2^k)\|^2)\\
&\leq\Upsilon_k+V_1(\E_k(\|x^{k+1}-x^k\|^2)+\|x^k-x^{k-1}\|^2)
\end{align}
and for $\Gamma_k=\sum_{i=1}^s\|v_k^i\|$,
\begin{align}
&\E_k(\|\tilde\nabla_{x_1}H(x_1^k,x_2^k)-\nabla_{x_1}H(x_1^k,x_2^k)\|+\|\tilde\nabla_{x_2}H(x_1^{k+1},x_2^k)-\nabla_{x_2}H(x_1^{k+1},x_2^k)\|)\\
&\leq \Gamma_k+V_2(\E_k(\|x^{k+1}-x^k\|)+\|x^k-x^{k-1}\|).
\end{align}
\item The sequence $(\Upsilon_k)_k$ decays geometrically, that is
\begin{align}
\E_k(\Upsilon_{k+1})\leq(1-\rho)\Upsilon_k+V_\Upsilon(\E_k(\|x^{k+1}-x^k\|^2)+\|x^k-x^{k-1}\|^2). \label{eq:variance_reduced_geom_decay}
\end{align}
\item If $\lim_{k\to\infty}\E(\|x^k-x^{k-1}\|^2)=0$, then it holds $\E(\Upsilon_k)\to0$ and $\E(\Gamma_k)\to0$ as $k \rightarrow \infty$.
\end{enumerate}
\end{Definition}

While the SGD estimator is not variance-reduced, many popular gradient estimators as the 
SAGA \cite{DBL2014} and SARAH \cite{NLST2017} estimators have this property.
Since for many problems in image processing and machine learning 
the SAGA estimator is not applicable due to its high memory requirements,
we will focus on the SARAH estimator in this paper. 

\begin{Definition} [SARAH Estimator] \label{def:sarah}
The SARAH \emph{estimator} reads for $k=0$ as
$$
\tilde\nabla_{x_1} H(x_1^0,x_2^0)=\nabla_{x_1} H(x_1^0,x_2^0).
$$
For $k=1,2,...$ we define random variables $p_i^k\in\{0,1\}$ with $P(p_i^k=0)=\tfrac1p$ and $P(p_i^k=1)=1-\tfrac1p$, 
where $p\in(1,\infty)$ is a fixed chosen parameter. Further, we define $B_i^k$ to be random subsets 
uniformly drawn from $\{1,...,n\}$ of fixed batch size $b$. 
Then for $k=1,2,...$ the SARAH estimator reads as
{\small
\begin{equation}\label{sarah}
\tilde\nabla_{x_1} H(x_1^k,x_2^k) =
\begin{cases}
\nabla_{x_1} H(x_1^k,x_2^k),&$if $p_1^k=0,\\
\tfrac1b \sum\limits_{i\in B_i^k}\nabla_{x_1} h_i(x_1^k,x_2^k)-\nabla_{x_1} h_i(x_1^{k-1},x_2^{k-1}) 
+ \tilde\nabla_{x_1}H(x_1^{k-1},x_2^{k-1}),&$if $p_1^k=1,
\end{cases}
\end{equation}
}
and analogously for $\tilde \nabla_{x_2} H$. In the sequel, 
we assume that the family of the random elements $p_i^k$, $B_i^k$ for $i=1,2$ and $k=1,2,...$ is independent.
\end{Definition}

The following proposition was shown in \cite{DTLDS2020}.

\begin{Proposition}
Let $H\colon\R^{d_1} \times \R^{d_2} \to \R$ be given by \eqref{special_form_h} with
functions  $h_i\colon \R^{d_1}\times\R^{d_2}$ having a globally $M$-Lipschitz continuous gradient.
Then the SARAH gradient estimator is variance-reduced with parameters $\rho=\tfrac1p$ and $V_1=V_\Upsilon=2M^2$, $V_2=2M$.
\end{Proposition}

The convergence results in the next theorem were proven in \cite{DTLDS2020}. 
We refer to the type of convergence in \eqref{spring_conv_2} as \emph{linear convergence}.
Note that the parameter $V_2$ from the definition of variance reductions does not appear in the theorem.
Actually, Driggs et al.  \cite{DTLDS2020} need the assumption containing $V_2$ to prove tighter convergence rates 
for semi-algebraic functions $F$.

\begin{Theorem}[Convergence of SPRING]\label{thm:conv_spring}
Let $F$ by given as in \eqref{eq:PALM_min_general}, where
$H$ fulfills the Assumptions \ref{ass:PALM2}.
Let $\tilde\nabla$ be a variance-reduced estimator for $H$ with parameters $V_1,V_\Upsilon\geq0$ and $\rho\in(0,1]$.
Assume that 
$\bar\gamma_k\coloneqq\max(\tfrac1{\tau_1^k},\tfrac1{\tau_2^k})$ 
is non-increasing and that  $0<\beta \coloneqq \inf_k \min( \tfrac1{\tau_1^k},\tfrac1{\tau_2^k})$.
Further suppose that for all $k\in\N$,
$$
\bar\gamma_k\leq\tfrac1{16}\sqrt{\tfrac{M^2}{(V_1+\tfrac{V_\Upsilon}{\rho})^2}
+\tfrac{16}{V_1+\tfrac{V_\Upsilon}\rho}}-\tfrac{M}{16(V_1+\tfrac{V_1}\rho)}.
$$
Let the stepsize in SPRING fulfill
$\tau_1^k>4 \lambda_1^+$, $\tau_2^k>4\lambda_2^+$ and set $\eta\coloneqq\max(\tfrac{\tau_1^k}{4}-\lambda_1^+,\tfrac{\tau_2^k}{4}-\lambda_2^+)$.
Then, with $t$ drawn uniformly from $\{0,...,T-1\}$, the generalized gradient at $(x_1^t,x_2^t)$ after $T$ iterations of SPRING satisfies
\begin{align} \label{spring_conv_1}
\E(\|\mathscr{G}F_{2\tau_1^t,2\tau_2^t}(x_1^t,x_2^t)\|^2)\leq \frac{4 (F(x_1^0,x_2^0)+\tfrac{2\bar\gamma_0}{\rho}\Upsilon_0)}{T \eta \beta^2}.
\end{align}
Furthermore, if for some $\gamma>0$, the function $F$ fulfills the error bound
\begin{align}
F(x_1,x_2)-\ubar F\leq \gamma\|\mathscr{G}F_{\tau_1,\tau_2}(x_1,x_2)\|^2
\end{align} 
for all $(x_1,x_2) \in \R^{d_1} \times \R^{d_2}$, then it holds after $T$ iterations of SPRING that
\begin{align}\label{spring_conv_2}
\E(F(x_1^T,x_2^T)-\ubar F)\leq (1-\Theta)^T (F(x_1^0,x_2^0)-\ubar F+\tfrac{4\bar\gamma_0}{\rho}\Upsilon_0),
\end{align}
where $\Theta=\min(\tfrac{\gamma \eta \beta^2}{4},\tfrac\rho2)$. In particular, the sequence $\E(F(x_1^T,x_2^T))$ converges linearly to $\ubar F$ as $T\to\infty$.
\end{Theorem}

Finally, we like to mention that Davis et al. \cite{DEU2016} 
considered an asynchronous variant of PALM with stochastic noise in the gradient. 
Their approach requires an explicit bound on the noise, which is not fulfilled for the above gradient estimators. 
Thus, focus and setting in \cite{DEU2016} differ from those of SPRING.

%----------------------------------------------
\subsection{iSPALM}
%----------------------------------------------
In particular for machine learning and deep learning, 
the combination of momentum-like methods and a stochastic gradient estimator turned out to be essential, as discussed e.g.\ in \cite{GLZX2019,SMDH2013}.
Thus, inspired by the inertial PALM, we propose the inertial stochastic PALM (iSPALM) algorithm 
outlined in Algorithm \ref{alg:iSPRING}. 

\begin{algorithm}[!ht]
\caption{Inertial Stochastic Proximal Alternating Linearized Minimization (iSPALM)}\label{alg:iSPRING}
\begin{algorithmic}
\State Input:  $(x_1^{-1},x_2^{-1}) = (x_1^{0},x_2^{0})\in\R^{d_1}\times \R^{d_2}$, 
       parameters $\alpha_1^{k},\alpha_2^{k},\beta_1^{k},\beta_2^{k},\tau_1^{k},\tau_2^{k}$ for $k\in \N_0$.
\For {$k= 0,1,...$}
\State Set
\begin{align}
y_1^{k}&=x_1^{k}+\alpha_1^{k}(x_1^{k}-x_1^{k-1})\\
z_1^{k}&=x_1^{k}+\beta_1^{k}(x_1^{k}-x_1^{k-1})\\
x_1^{k+1}&\in\prox_{\tau_1^{k}}^f\big(y_1^{k}-\tfrac1{\tau_1^{k}}\tilde\nabla_{x_1} H(z_1^{k},x_2^{k})\big)
\end{align}
\State Set
\begin{align}
y_2^{k}&=x_2^{k}+\alpha_2^{k}(x_2^{k}-x_2^{k-1})\\
z_2^{k}&=x_2^{k}+\beta_2^{k}(x_2^{k}-x_2^{k-1})\\
x_2^{k+1}&\in\prox_{\tau_2^{k}}^g\big(y_2^{k}-\tfrac1{\tau_2^{k}}\tilde\nabla_{x_2} H(x_1^{k+1},z_2^{k})\big)
\end{align}
\EndFor
\end{algorithmic}
\end{algorithm}

\begin{Remark}
Similarly as in Remark \ref{rem_ipalm_momentum}, iSPALM can be viewed as a generalization of the stochastic versions of the momentum method and NAG to the nonsmooth case.
Note, that in the stochastic setting the theoretical error bounds of momentum methods are not tighter than for a plain gradient descent. An overview over these convergence results can be found in \cite{GLZX2019,SMDH2013}.
Consequently, we are not able to show tighter convergence rates for iSPALM than for stochastic PALM.
Nevertheless, stochastic momentum methods as the momentum SGD and the Adam optimizer \cite{KB2014} are widely used
and have shown a better convergence behavior than a plain SGD in a huge number of applications.
\hfill $\Box$
\end{Remark}

To prove that the generalized gradients on the sequence of iterates produced by iSRING convergence 
to zero, some properties of the gradient estimator are required.
The authors of \cite{DTLDS2020} assumed that the estimators are evaluated at $(x_1^k,x_2^k)$ and $(x_1^{k+1},x_2^k)$,  $k\in \N_0$. 
In contrast, we require that the gradient estimators are evaluated at $(z_1^k,x_2^k)$ and $(x^{k+1}_1,z_2^k)$ for $k\in \N_0$.
To prove a counterpart of Theorem \ref{thm:conv_spring}, we modify Definition \ref{def_var-red}. 
In particular, we need only the first part in (i).

\begin{Definition}[Inertial Variance-Reduced Gradient Estimator]\label{def:ivr}
A gradient estimator $\tilde\nabla$ is called \emph{inertial variance-reduced} 
for a differentiable function $H: \R^{d_1} \times \R^{d_2} \to \R$ 
with constants $V_1,V_\Upsilon\geq0$ and $\rho\in(0,1]$, 
if for any sequence $(x^k)_k=(x_1^k,x_2^k)_{k \in \N_0}$, $x^{-1} \coloneqq x^0$ and
any $0 \leq \beta_i^k < \bar \beta_i$, $i=1,2$ 
there exists a sequence of random variables $(\Upsilon_k)_{k \in \N}$ with $\E(\Upsilon_1) < \infty$ such that following 
holds true:
\begin{enumerate}
\item 
For 
$z_i^{k}\coloneqq x_i^{k}+\beta_i^{k}(x_i^{k}-x_i^{k-1})$, $i=1,2$,
we have
\begin{align}
&\E_k(\|\tilde\nabla_{x_1}H(z_1^k,x_2^k)-\nabla_{x_1}H(z_1^k,x_2^k)\|^2+\|\tilde\nabla_{x_2}H(x_1^{k+1},z_2^k)
-\nabla_{x_2}H(x_1^{k+1},z_2^k)\|^2)
\\
&\leq 
\Upsilon_k + V_1 \left(\E_k(\|x^{k+1}-x^k\|^2)+\|x^k-x^{k-1}\|^2+\|x^{k-1}-x^{k-2}\|^2 \right).
\end{align}
\item 
The sequence $(\Upsilon_k)_k$ decays geometrically, that is
\begin{align}
\E_k(\Upsilon_{k+1})\leq(1-\rho)\Upsilon_k+V_\Upsilon(\E_k(\|x^{k+1}\!-\!x^k\|^2)\!+\!\|x^k\!-\!x^{k-1}\|^2
\!+\!\|x^{k-1}-x^{k-2}\|^2). 
\end{align}
\item If $\lim_{k\to\infty}\E(\|x^k-x^{k-1}\|^2)=0$, then $\E(\Upsilon_k)\to 0$ as $k \rightarrow \infty$.
\end{enumerate}
\end{Definition}

To prove that the SARAH gradient estimator is inertial variance-reduced and that \mbox{iSPALM} converges, 
we need the following auxiliary lemma,
which can be proved analogously to \cite[Proposition 4.1]{PS2016}.

\begin{Lemma}\label{prop:deltas}
Let $(x_1^k,x_2^k)_k$ be an arbitrary sequence and $\alpha_i^k,\beta_i^k\in\R$, $i=1,2$.
Further define
$$
y_i^{k}\coloneqq x_i^{k}+\alpha_i^{k}(x_i^{k}-x_i^{k-1}),\quad
z_i^{k}\coloneqq x_i^{k}+\beta_i^{k}(x_i^{k}-x_i^{k-1}), \quad i=1,2,
$$
and
$$
\Delta_i^k \coloneqq \tfrac12\|x_i^k-x_i^{k-1}\|^2, \qquad i=1,2.
$$
Then, for any $k\in\N$ and $i=1,2$, we have
\end{Lemma}
\begin{enumerate}
\item $\|x_i^k-y_i^k\|^2=2(\alpha_i^k)^2\Delta_i^k$,
\item $\|x_i^k-z_i^k\|^2=2(\beta_i^k)^2\Delta_i^k$,
\item $\|x_i^{k+1}-y_i^k\|^2\geq 2(1-\alpha_i^k \Delta_i^{k+1}+2\alpha_i^k)(\alpha_i^k-1)\Delta_i^k$.
\end{enumerate}
%\end{Lemma}

Now we can show the desired property of the SARAH gradient estimator.

\begin{Proposition}\label{prop:sarah-ivr}
Let $H\colon\R^{d_1} \times \R^{d_2} \to \R$ be given by \eqref{special_form_h} with
functions  $h_i\colon \R^{d_1}\times\R^{d_2}$ having a globally $M$-Lipschitz continuous gradient.
Then the SARAH estimator $\tilde\nabla$ is inertial variance-reduced with parameters $\rho=\tfrac1p$ and
$$
%{\color{red} V_1} = 
V_\Upsilon = 3(1-\tfrac1p)M^2 \left(1+\max\left((\bar\beta_1)^2,(\bar\beta_2)^2\right) \right).
$$
Furthermore, we can choose 
$$
\Upsilon_{k+1} = 
\|\tilde\nabla_{x_1}H(z_1^k,x_2^k)-\nabla_{x_1}H(z_1^k,x_2^k)\|^2
+\|\tilde\nabla_{x_2}H(x_1^{k+1},z_2^k)-\nabla_{x_2}H(x_1^{k+1},z_2^k)\|^2.
$$
\end{Proposition}

The proof is given in Appendix \ref{app:sarah}.

%---------------------------------------------------------
\section{Convergence Analysis of iSPALM} \label{sec:conv}
%---------------------------------------------------------
We assume that the parameters of iSPALM fulfill the following conditions.

\begin{Assumption}[Conditions on the Parameters of iSPALM]\label{ass:iSPRING2}
Let $\lambda_i^+$, $i=1,2$ and $L_1(x_2^k)$, $L_2(x_1^k)$ be defined by Assumption \ref{ass:PALM2}
and $\rho, V_1,V_\Upsilon$ by Definition \ref{def:ivr}. 
Further, let $H\colon\R^{d_1} \times \R^{d_2} \to \R$ be given by \eqref{special_form_h} with
functions  $h_i\colon \R^{d_1}\times\R^{d_2}$ having a globally $M$-Lipschitz continuous gradient.
There exist $\epsilon,\varepsilon>0$ such that for all $k\in\N$ and $i=1,2$ the following holds true:
\begin{enumerate}[(i)]
\item  
There exist $0<\bar\alpha_i<\tfrac{1-\epsilon}{2}$ such that $0\leq\alpha_i^{k}\leq\bar\alpha_i$
and $0<\bar\beta_i\le 1$ such that $0\leq\beta_i^{k}\leq\bar\beta_i$
\item 
The parameters $\tau_i^{k}$, $i=1,2$ are given by 
$$
\tau_1^{k}
\coloneqq \frac{(1+\epsilon)\delta_1+M+L_1(x_2^k)+S}{1-\alpha_1^k}, 
\quad \mathrm{and} \quad
\tau_2^{k}
\coloneqq \frac{(1+\epsilon)\delta_2+M+L_2(x_1^{k+1})+S}{1-\alpha_2^k},
$$
where $S := 4 \tfrac{\rho V_1+V_\Upsilon}{\rho M}+\varepsilon$ 
and for $i=1,2$,
$$
\delta_i \coloneqq \frac{(M+\lambda_i^+)\bar\alpha_i+2\lambda_i^+\bar\beta_i^2+S}{1-2\bar\alpha_i-\epsilon}.
$$
\end{enumerate}
\end{Assumption}

To analyze the convergence behavior of iSPALM, we start with an auxiliary 
lemma which can be proven analogously to \cite[Lemma 3.2]{PS2016}.

\begin{Lemma}\label{lem:prox_ineq}
Let $\psi=\sigma+h$, where 
$h\colon\R^d\to\R$ is a continuously differentiable function with $L_h$-Lipschitz continuous gradient, and 
$\sigma\colon\R^d\to(-\infty,\infty]$ is proper and lower semicontinuous with $\inf_{\R^d}\sigma>-\infty$. 
Then it holds for any $u,v,w\in\R^d$ and any $u^+\in\R^d$ defined by
$$
u^+\in\prox_t^\sigma(v-\tfrac1t \tilde\nabla h(w)),\quad t>0
$$
that
$$
\psi(u^+)\leq \psi(u)+\langle u^+-u,\nabla h(u)-\tilde\nabla h(w)\rangle + \frac{L_h}2^2  \|u - u^+\|^2
+\frac{t}2\|u-v\|^2-\frac{t}2\|u^+-v\|^2.
$$
\end{Lemma}

Now we can establish a result on the expectation of squared subsequent iterates. Note that equivalent results were shown for PALM, iPALM and SPRING. Here we use a function $\Psi$, which not only contains the current function value, but also the distance of the iterates to the previous ones. 
A similar idea was used in the convergence proof of iPALM \cite{PS2016}. 
Nevertheless, incorporating the stochastic gradient estimator here makes the proof much more involved.

\begin{Theorem}\label{thm:l2_steps}
Let $F\colon \R^{d_1} \times \R^{d_2}\to (-\infty,\infty]$ be given by \eqref{eq:PALM_min_general} and
fulfill Assumption \ref{ass:PALM2}.
Let $(x_1^k,x_2^k)_k$ be generated by iSPALM with parameters fulfilling Assumption \ref{ass:iSPRING2},
where we use an inertial variance-reduced gradient estimator $\tilde\nabla$.
Then it holds for  $\Psi\colon (\R^{d_1}\times\R^{d_2})^3 \to \R$ defined for
$u = (u_{11},u_{12},u_{21},u_{22},u_{31},u_{32}) \in(\R^{d_1}\times\R^{d_2})^3$ by
$$
\Psi(u)\coloneqq F(u_{11},u_{12})+\tfrac{\delta_1}{2}\|u_{11}-u_{21}\|^2+\tfrac{\delta_2}{2}\|u_{12}-u_{22}\|^2
+\tfrac{S}{4} \left( \|u_{21}-u_{31}\|^2+\|u_{22}-u_{32}\|^2\right)
$$
that there exists $\gamma>0$ such that
\begin{equation}
\Psi(u^1)-\inf_{u\in(\R^{d_1}\times\R^{d_2})^2}\Psi(u)+\tfrac1{M\rho}\E(\Upsilon_1) 
\geq 
\gamma \sum_{k=0}^T\E(\|u^{k+1}-u^k\|^2),
\end{equation}
where
$u^k \coloneqq (x_1^k,x_2^k,x_1^{k-1},x_2^{k-1},x_1^{k-2},x_2^{k-2} )$.
In particular, we have 
\begin{align}
\sum_{k=0}^\infty\E(\|u^{k+1}-u^k\|^2)<\infty.
\end{align}
\end{Theorem}

\begin{proof}
By Lemma \ref{lem:prox_ineq} with $\psi \coloneqq H(\cdot,x_2) + f$, we obtain
\begin{align}
H(x_1^{k+1},x_2^k)+f(x_1^{k+1})
&\leq H(x_1^k,x_2^k)+f(x_1^k)\\
& \quad + \left\langle x_1^{k+1}-x_1^k,\nabla_{x_1} H(x_1^k,x_2^k)-\tilde\nabla_{x_1} H(z_1^k,x_2^k) \right\rangle
\\
& \quad + \tfrac{L_1(x_2^k)}{2} \|x_1^{k+1}-x_1^k\|^2+\tfrac{\tau_1^k}2\|x_1^k-y_1^k\|^2-\tfrac{\tau_1^k}2\|x_1^{k+1}-y_1^k\|^2. \label{eq1}
\end{align}
Using $ab \le \frac{s}{2} a^2 + \frac{1}{2s} b^2$ for $s>0$  and $\|a-c\|^2\leq2\|a-b\|^2+2\|b-c\|^2$ 
the inner product is smaller or equal than
\begin{align}
&\tfrac{s_1^k}2\|x_1^{k+1}-x_1^k\|^2+\tfrac1{2s_1^k}\|\nabla_{x_1} H(x_1^k,x_2^k)-\tilde\nabla_{x_1} H(z_1^k,x_2^k)\|^2\\
&\leq \tfrac{s_1^k}2\|x_1^{k+1}-x_1^k\|^2+\tfrac1{s_1^k}\|\nabla_{x_1} H(z_1^k,x_2^k)-\tilde\nabla_{x_1} H(z_1^k,x_2^k)\|^2\\
&+\tfrac1{s_1^k}\|\nabla_{x_1} H(x_1^k,x_2^k)-\nabla_{x_1} H(z_1^k,x_2^k)\|^2
\\
&=\tfrac{s_1^k}2\|x_1^{k+1}-x_1^k\|^2+\tfrac1{s_1^k}\|\nabla_{x_1} H(z_1^k,x_2^k)
-\tilde\nabla_{x_1} H(z_1^k,x_2^k)\|^2+\tfrac{L_1(x_2^k)^2}{s_1^k}\|x_1^k - z_1^k\|^2.
\end{align}
Combined with \eqref{eq1} this becomes 
\begin{align}
&H(x_1^{k+1},x_2^k)+f(x_1^{k+1})\\
&\leq H(x_1^k,x_2^k)+f(x_1^k) + \tfrac{L_1(x_2^k)}{2} \|x_1^{k+1}-x_1^k\|^2+\tfrac{\tau_1^k}2\|x_1^k-y_1^k\|^2-\tfrac{\tau_1^k}2\|x_1^{k+1}-y_1^k\|^2\\
&+\tfrac{s_1^k}2\|x_1^{k+1}-x_1^k\|^2+\tfrac1{s_1^k}\|\nabla_{x_1} H(z_1^k,x_2^k)-\tilde\nabla_{x_1} H(z_1^k,x_2^k)\|^2
+\tfrac{L_1(x_2^k)^2}{s_1^k}\|x_1^k - z_1^k\|^2.
\end{align}
Using Lemma \ref{prop:deltas} we get
\begin{align}
&H(x_1^{k+1},x_2^k)+f(x_1^{k+1})\\
&\leq H(x_1^k,x_2^k)+f(x_1^k)+ \left(L_1(x_2^k)+s_1^k-\tau_1^k(1-\alpha_1^k) \right)\Delta_1^{k+1}
\\
&+\tfrac1{s_1^k} \left( 2L_1(x_2^k)^2(\beta_1^k)^2+s_1^k\tau_1^k\alpha_1^k \right)\Delta_1^k
+\tfrac1{s_1^k}\|\nabla_{x_1} H(z_1^k,x_2^k)-\tilde\nabla_{x_1} H(z_1^k,x_2^k)\|^2.
\end{align}
Analogously we conclude for $\psi \coloneqq H(x_1,\cdot) + g$ that
\begin{align}
&H(x_1^{k+1},x_2^{k+1})+g(x_2^{k+1})\\
&\leq H(x_1^{k+1},x_2^k)+g(x_2^k)+ \left(L_2(x_1^{k+1})+s_2^k-\tau_2^k(1-\alpha_2^k) \right)\Delta_2^{k+1}\\
&+\tfrac1{s_2^k} \left(2L_2(x_1^{k+1})^2(\beta_2^k)^2+s_2^k\tau_2^k\alpha_2^k \right)\Delta_2^k 
+\tfrac1{s_2^k}\|\nabla_{x_2} H(x_1^{k+1},z_2^k)-\tilde\nabla_{x_2} H(x_1^{k+1},z_2^k)\|^2.
\end{align}
Adding the last two inequalities and using the abbreviation
$L_1^k \coloneqq L_1(x_2^k)$ and $L_2^k \coloneqq L_2(x_1^{k+1})$,
we obtain
\begin{align}
&F(x_1^{k+1},x_2^{k+1}) \leq F(x_1^k,x_2^k)\\
&+ \sum_{i=1}^2 \left( \left(L_i^k+s_i^k-\tau_i^k(1-\alpha_i^k)\right)\Delta_i^{k+1}
+\tfrac1{s_i^k}\left(2 (L_i^k)^2(\beta_i^k)^2+s_i^k\tau_i^k\alpha_i^k\right)\Delta_i^k \right)\\
&+\tfrac1{s_1^k}\|\nabla_{x_1} H(z_1^k,x_2^k)-\tilde\nabla_{x_1} H(z_1^k,x_2^k)\|^2
+\tfrac1{s_2^k}\|\nabla_{x_2} H(x_1^{k+1},z_2^k)-\tilde\nabla_{x_2} H(x_1^{k+1},z_2^k)\|^2 .\label{eq2}
\end{align}
Reformulating \eqref{eq2} in terms of 
\begin{equation} \label{**}
\Psi(u^k) = F(x_1^k,x_2^k)+\delta_1 \Delta_1^k+\delta_2\Delta_2^k+ \tfrac{S}2(\Delta_1^{k-1}+\Delta_2^{k-1})
\end{equation}
leads to 
\begin{align}
&\Psi(u^k)-\Psi(u^{k+1})
=F(x_1^k,x_2^k)-F(x_1^{k+1},x_2^{k+1})+\delta_1\Delta_1^k+\delta_2\Delta_2^k-\delta_1\Delta_1^{k+1}-\delta_2\Delta_2^{k+1}\\
&+ \tfrac{S}2\left(\Delta_1^{k-1} + \Delta_2^{k-1} - \Delta_1^k - \Delta_2^k\right)\\
&\geq 
\sum_{i=1}^2 \left( \left(\tau^k_i(1-\alpha_i^k)-s_i^k-L_i^k-\delta_i \right)\Delta_i^{k+1}\right)
+ \sum_{i=1}^2 \left( \left(\delta_i-\tfrac2{s_i^k} (L_i^k)^2 (\beta_i^k)^2-\tau_i^k\alpha_i^k\right )\Delta_i^k\right)\\
&- \tfrac1{s_1^k}\|\nabla_{x_1} H(z_1^k,x_2^k)-\tilde\nabla_{x_1} H(z_1^k,x_2^k)\|^2
- \tfrac1{s_2^k}\|\nabla_{x_2} H(x_1^{k+1},z_2^k)-\tilde\nabla_{x_2} H(x_1^{k+1},z_2^k)\|^2\\
&+ \tfrac{S}2\left(\Delta_1^{k-1} + \Delta_2^{k-1} - \Delta_1^k - \Delta_2^k\right).\label{eq3}
\end{align}
Now, we set $s_1^k=s_2^k\coloneqq M$ use that $L_i^k\le M$, take the conditional expectation $\E_k$ in \eqref{eq3} 
and use that $\tilde\nabla$ is an inertial variance-reduced estimator to get 
\begin{align}
&\Psi(u^k)-\E_k(\Psi(u^{k+1}))\\
&\geq 
\sum_{i=1}^2 \left( \left(\tau^k_i(1-\alpha_i^k)\!-\! M\!-\! L_i^k-\delta_i \right)\E_k(\Delta_i^{k+1})
+ 
\left(\delta_i-\tfrac2{M} (L_i^k)^2(\beta_i^k)^2-\tau_i^k\alpha_i^k\right)\Delta_i^k \right)\\
&
- \tfrac{2V_1}{M} \sum_{i=1}^2 \left( \E_k(\Delta_i^{k+1})+\Delta_i^k \right) - \tfrac1{M}\Upsilon_k
+  \tfrac{S}2\left(\Delta_1^{k-1} + \Delta_2^{k-1}- \Delta_1^k - \Delta_2^k\right)
\\
&\ge 
\sum_{i=1}^2 \left(
\left(\tau^k_i(1-\alpha_i^k)-M-L_i^k-\delta_i-\tfrac{2V_1}{M}\right)\E_k(\Delta_i^{k+1}) \right)
\\
&+ 
\sum_{i=1}^2 \left(\left(\delta_i-2 L_i^k(\beta_i^k)^2-\tau_i^k\alpha_i^k-\tfrac{2V_1}{M}\right)\Delta_i^k \right)
\\
&- \tfrac1{M}\Upsilon_k
+ \tfrac{S}2\left(\Delta_1^{k-1} + \Delta_2^{k-1}- \Delta_1^k - \Delta_2^k\right).\label{eq4}
\end{align}
Since $\tilde\nabla$ is inertial variance-reduced, we know from Definition \ref{def:ivr} (ii) that
\begin{align} \label{eq:upsilon}
\rho \Upsilon_k
\leq
\Upsilon_k-\E_k(\Upsilon_{k+1})
+
2V_\Upsilon \sum_{i=1}^2 \left(\E_k(\Delta_i^{k+1})+\Delta_i^k + \Delta_i^{k-1}\right).
\end{align}
Inserting this in \eqref{eq4} and using the definition of $S$ yields 
\begin{align}
&\Psi(u^k)-\E_k\left(\Psi(u^{k+1}) \right)
\geq 
\sum_{i=1}^2 \left(\left(\tau^k_i (1-\alpha_i^k)-M-L_i^k)-\delta_i- \tfrac{S}2 \right)\E_k(\Delta_i^{k+1}) \right)\\
&+ 
\sum_{i=1}^2 \left( \left(\delta_i-2 L_i^k (\beta_i^k)^2-\tau_i^k\alpha_i^k-\tfrac{S}2 \right)\Delta_i^k \right)\\
&- \frac{2 V_\Upsilon}{\rho M} (\Delta_1^{k-1} + \Delta_2^{k-1}) + \tfrac{1}{M\rho} \left(\E_k(\Upsilon_{k+1})-\Upsilon_k \right)
+ \tfrac{S}2\left(\Delta_1^{k-1} + \Delta_2^{k-1}- \Delta_1^k - \Delta_2^k\right)\\ 
&\geq 
\sum_{i=1}^2 
\Big( 
\underbrace{
\left(\tau^k_i (1-\alpha_i^k)-M-L_i^k -\delta_i- S \right)}_{a^k_i} 
\E_k(\Delta_i^{k+1}) 
\Big)\\
&+ 
\sum_{i=1}^2 \Big(
\underbrace{
\left(\delta_i- 2 L_i^k(\beta_i^k)^2-\tau_i^k\alpha_i^k-S \right)}_{b^k_i}
\Delta_i^k 
\Big)\\
&+ \tfrac{1}{M\rho} \left(\E_k(\Upsilon_{k+1})-\Upsilon_k \right)
+ \left(\tfrac{S}2 - \tfrac{2 V_\Upsilon}{\rho M}\right)\left(\Delta_1^{k-1} + \Delta_2^{k-1}\right). \label{fast}
\end{align}
Choosing  $\tau_i^k$, $\delta_i$, $i=1,2$ and $\epsilon$ 
as in Assumption \ref{ass:iSPRING2}(ii), 
we obtain by straightforward computation for $i=1,2$ and all $k \in \mathbb N$ that
$a_i^k = \epsilon \delta_i$ and
\begin{align}
b_i^k
&=
\tfrac{1}{1-\alpha_i^k}
\left( (1- \epsilon - 2 \alpha_i^k) \delta_i - \alpha_i^k M - S - L_i^k \left( 2(\beta_i^k)^2(1-\alpha_i^k) + \alpha_i^k \right) \right)
+\epsilon \delta_i\\
&\ge
\tfrac{1}{1-\alpha_i^k} \left
((1- \epsilon - 2 \bar \alpha_i) \delta_i - \bar \alpha_i M - S - \lambda_i^
+ \left( 2(\bar \beta_i)^2(1-\alpha_i^k) + \bar \alpha_i \right)\right)
+\epsilon \delta_i\\
&= 
\epsilon \delta_i + 2 \frac{2 \lambda_i^+\alpha_i^k (\bar \beta_i)^2 }{1-\alpha_i^k} \ge \epsilon \delta_i.
\end{align} 
Applying this in \eqref{fast}, we get 
\begin{align}
\Psi(u^k)-\E_k\left(\Psi(u^{k+1}) \right)
&\geq
\epsilon\min(\delta_1,\delta_2) \sum_{i=1}^2 (\E_k(\Delta_i^{k+1})+\Delta_i^k)\\
&+\tfrac1{M\rho}(\E_k(\Upsilon_{k+1})-\Upsilon_k)
+ \left(\tfrac{S}{2}- \tfrac{2 V_\Upsilon}{\rho M}\right)\left(\Delta_1^{k-1} + \Delta_2^{k-1}\right).
\end{align}
By definition of $S$ it holds $\left(\frac{2 V_\Upsilon}{\rho M}-\tfrac{S}2\right)\geq\varepsilon$. 
Thus, we get for $\gamma \coloneqq \tfrac12 \min(\epsilon \delta_1,\epsilon \delta_2,\varepsilon)$ that
\begin{align}
\Psi(u^k)-\E_k\left(\Psi(u^{k+1}) \right)
&\geq
2\gamma \sum_{i=1}^2 \left(\E_k(\Delta_i^{k+1})+\Delta_i^k+ \Delta_i^{k-1} \right) 
+\tfrac1{M\rho}(\E_k(\Upsilon_{k+1})-\Upsilon_k).
\end{align}
Taking the full expectation yields
\begin{align}\label{eq:exp_psi}
\E(\Psi(u^k)-\Psi(u^{k+1})) &\geq \gamma \E(\|u^{k+1}-u^k\|^2)+\tfrac1{M\rho}\E(\Upsilon_{k+1}-\Upsilon_k),
\end{align}
and summing up for $k=1,...,T$, 
\begin{align}
\E(\Psi(u^1)-\Psi(u^{T+1})) &\geq  \gamma \sum_{k=0}^T\E(\|u^{k+1}-u^k\|^2)+\tfrac1{M\rho}\E(\Upsilon_{T+1}-\Upsilon_1).
\end{align}
Since $\Upsilon_k\geq 0$, this yields
\begin{align}
\gamma 
\sum_{k=0}^T\E(\|u^{k+1}-u^k\|^2) \le
\Psi(u^1)-\underbrace{\inf_{u\in(\R^{d_1}\times\R^{d_2})^2}\Psi(u)}_{>-\infty}+\underbrace{\tfrac1{M\rho}\E(\Upsilon_1)}_{<\infty}
< 
\infty.
\end{align}
This finishes the proof.
\end{proof}

Next, we want relate the sequence of iterates generated by iSPALM to the subgradient of the objective function. 
Such a relation was also established for the (inertial) PALM algorithm. 
However, due to the stochastic gradient estimator the proof differs significantly from its deterministic counterparts.
Note that the convergence analysis of SPRING in \cite{DTLDS2020} does not use the subdifferential but the so-called generalized gradient $\mathscr{G}F_{\tau_1,\tau_2}$.
This is not satisfying at all, since it becomes not clear how this generalized gradient is related to the (sub)differential of the objective function in limit processes with varying $\tau_1$ and $\tau_2$. 
In particular, it is easy to find examples of $F$ and sequences $(\tau_1^k)_k$ and $(\tau_2^k)_k$ such that the generalized gradient $\mathscr{G}F_{\tau_1^k,\tau_2^k}(x_1,x_2)$ is non-zero, but converges to zero for fixed $x_1$ and $x_2$.

\begin{Theorem}\label{thm:grad_decay}
Under the assumptions of Theorem \ref{thm:l2_steps} there exists some $C>0$ such that
\begin{align}
\E\left(\dist(0,\partial F(x_1^{k+1},x_2^{k+1}))^2\right)\leq C \E(\|u^{k+1}-u^k\|^2)+3\E(\Upsilon_k).
\end{align}
In particular, it holds
\begin{align}
\E\left(\dist(0,\partial F(x_1^{k+1},x_2^{k+1}))^2 \right) \to 0\quad\text{as }k\to\infty.
\end{align}
\end{Theorem}

\begin{proof}
By definition of $x_1^{k+1}$, and \eqref{star} as well as Proposition \ref{prop:subgrad} it holds
\begin{align}
0\in \tau_1^k(x_1^{k+1}-y_1^k)+\tilde\nabla_{x_1} H(z_1^k,x_2^k)+\partial f(x_1^{k+1}).
\end{align}
This is equivalent to
\begin{align}
&\tau_1^k(y_1^k-x_1^{k+1})+\nabla_{x_1}H(x_1^{k+1},x_2^{k+1})-\tilde\nabla_{x_1}H(z_1^k,x_2^k)\\
&\in\nabla_{x_1} H(x_1^{k+1},x_2^{k+1})+\partial f(x_1^{k+1}) \in \partial_{x_1} F(x_1^{k+1},x_2^{k+1}).
\end{align}
Analogously we get that
\begin{align}
&\tau_2^k(y_2^k-x_2^{k+1})+\nabla_{x_2}H(x_1^{k+1},x_2^{k+1})-\tilde\nabla_{x_1}H(x_1^{k+1},z_2^k)\\
&\in\nabla_{x_2} H(x_1^{k+1},x_2^{k+1})+\partial g(x_2^{k+1}) \in \partial_{x_2} F(x_1^{k+1},x_2^{k+1}).
\end{align}
Then we obtain by Proposition \ref{prop:subgrad} that 
\begin{align}
v \coloneqq
\left(\begin{array}{c}\tau_1^k(y_1^k-x_1^{k+1})+\nabla_{x_1}H(x_1^{k+1},x_2^{k+1})-\tilde\nabla_{x_1}H(z_1^k,x_2^k)\\\tau_2^k(y_2^k-x_2^{k+1})+\nabla_{x_2}H(x_1^{k+1},x_2^{k+1})-\tilde\nabla_{x_1}H(x_1^{k+1},z_2^k)\end{array}\right)\in \partial F(x_1^{k+1},x_2^{k+1}),
\end{align}
and it remains to show that the squared norm of $v$ is in expectation bounded by 
$C \E( \|u^{k+1}-u^k\|^2)+3\E(\Upsilon_k)$ for some $C>0$. 
Using $(a+b+c)^2 \le 3(a^2+b^2+c^2)$ we estimate
\begin{align}
\|v\|^2
&= \|\tau_1^k(y_1^k-x_1^{k+1})+\nabla_{x_1}H(x_1^{k+1},x_2^{k+1})-\tilde\nabla_{x_1}H(z_1^k,x_2^k)\|^2\\
&+\|\tau_2^k(y_2^k-x_2^{k+1})+\nabla_{x_2}H(x_1^{k+1},x_2^{k+1})-\tilde\nabla_{x_1}H(x_1^{k+1},z_2^k)\|^2\\
&\leq 3(\tau_1^k)^2\|y_1^k-x_1^{k+1}\|^2+3\|\nabla_{x_1}H(x_1^{k+1},x_2^{k+1})-\nabla_{x_1}H(z_1^{k},x_2^{k})\|^2\\
&+3\|\nabla_{x_1}H(z_1^{k},x_2^{k})-\tilde\nabla_{x_1}H(z_1^k,x_2^k)\|^2+3(\tau_2^k)^2\|y_2^k-x_2^{k+1}\|^2\\
&+3\|\nabla_{x_2}H(x_1^{k+1},x_2^{k+1})-\nabla_{x_2}H(x_1^{k+1},z_2^{k})\|^2\\
&+3\|\nabla_{x_2}H(x_1^{k+1},z_2^{k})-\tilde\nabla_{x_2}H(x_1^{k+1},z_2^{k})\|^2.
\end{align}
Since $\nabla H$ is $M$-Lipschitz continuous and  $(a+b)^2\leq 2(a^2+b^2)$, we get further
\begin{align}
\|v\|^2
&\le
12(\tau_1^k)^2\Delta_1^{k+1}+6(\tau_1^k)^2\|y_1^k-x_1^k\|^2+12(\tau_2^k)^2\Delta_2^{k+1}+6(\tau_2^k)^2\|y_2^k-x_2^k\|^2\\
&+3M^2\|x_1^{k+1}-z_1^k\|^2+6M^2\Delta_2^{k+1}+3M^2\|x_2^{k+1}-z_2^k\|^2\\
&+3 \left(\|\nabla_{x_1}H(z_1^{k},x_2^{k})-\tilde\nabla_{x_1}H(z_1^k,x_2^k)\|^2+\|\nabla_{x_2}H(x_1^{k+1},z_2^{k})
-\tilde\nabla_{x_2}H(x_1^{k+1},z_2^{k})\|^2\right).
\end{align}
Using Lemma \ref{prop:deltas} and the fact that $\tilde\nabla$ is inertial variance-reduced, this implies
\begin{align}
\|v\|^2
&\leq 12(\tau_1^k)^2\Delta_1^{k+1}+12(\tau_1^k)^2(\alpha_1^k)^2\Delta_1^k+12(\tau_2^k)^2\Delta_2^{k+1}+12(\tau_2^k)^2(\alpha_2^k)^2\Delta_2^k\\
&+12M^2\Delta_1^{k+1}+6M^2\|x_1^{k}-z_1^k\|^2+6M^2\Delta_2^{k+1}+12M^2\Delta_2^{k+1}+6M^2\|x_2^{k}-z_2^k\|^2\\
&+3\left( \|\nabla_{x_1}H(z_1^{k},x_2^{k})-\tilde\nabla_{x_1}H(z_1^k,x_2^k)\|^2+\|\nabla_{x_2}H(x_1^{k+1},z_2^{k})
-\tilde\nabla_{x_2}H(x_1^{k+1},z_2^{k})\|^2\right)\\
&\leq 
12\left( (\tau_1^k)^2+M^2 \right)\Delta_1^{k+1}
+12\left((\tau_1^k)^2(\alpha_1^k)^2+M^2(\beta_1^k)^2\right)\Delta_1^k\\
&
+\left(12(\tau_2^k)^2+18M^2\right)\Delta_2^{k+1}
+ 12\left((\tau_2^k)^2(\alpha_2^k)^2+M^2(\beta_2^k)^2\right)\Delta_2^k\\
&+3\left(\|\nabla_{x_1}H(z_1^{k},x_2^{k})-\tilde\nabla_{x_1}H(z_1^k,x_2^k)\|^2+\|\nabla_{x_2}H(x_1^{k+1},z_2^{k})
-\tilde\nabla_{x_2}H(x_1^{k+1},z_2^{k})\|^2\right)\\
&\leq C_0 \|u^{k+1}-u^k\|^2\\
&+3 (\|\nabla_{x_1}H(z_1^{k},x_2^{k})-\tilde\nabla_{x_1}H(z_1^k,x_2^k)\|^2)\\
&+3 (\|\nabla_{x_2}H(x_1^{k+1},z_2^{k}) -\tilde\nabla_{x_2}H(x_1^{k+1},z_2^{k})\|^2 ),
\end{align}
where 
$$
C_0=12 \max\left((\tau_1^k)^2+M^2,(\tau_1^k)^2(\alpha_1^k)^2+M^2(\beta_1^k)^2,(\tau_2^k)^2+\frac32 M^2,(\tau_2^k)^2(\alpha_2^k)^2+M^2(\beta_2^k)^2\right).
$$
Noting that $\dist(0,\partial F(x_1^{k+1},x_2^{k+1}))^2 \le \|v\|^2$, taking the conditional expectation $\E_k$ and using that $\tilde\nabla$ 
is inertial variance-reduced, we conclude
\begin{align}
&\E_k\left(\dist(0,\partial F(x_1^{k+1},x_2^{k+1}))^2\right)\\
&\leq \E_k\left(C_0\|u^{k+1}-u^k\|^2\right)\\
&+3\E_k\left(\|\nabla_{x_1}H(z_1^{k},x_2^{k})-\tilde\nabla_{x_1}H(z_1^k,x_2^k)\|^2+\|\nabla_{x_2}H(x_1^{k+1},z_2^{k})-\tilde\nabla_{x_2}H(x_1^{k+1},z_2^{k})\|^2\right)\\
&\leq\E_k\left((C_0+3V_1)\|u^{k+1}-u^k\|^2\right)+3\Upsilon_k.
\end{align}
Taking the full expectation on both sides and setting $C\coloneqq C_0+3V_1$ proves the claim.
\end{proof}

Using Theorem \ref{thm:grad_decay}, we can show the linear decay of the expected squared distance of the subgradient to $0$.

\begin{Theorem}[Convergence of iSPALM]\label{thm:ispring_conv_1}
Under the assumptions of Theorem \ref{thm:l2_steps} it holds for $t$ drawn uniformly from $\{2,...,T+1\}$ that there exists some $0<\sigma<\gamma$ such that
\begin{align}
\E\left(\dist(0,\partial F(x_1^t,x_2^t))^2\right)
&\leq
\tfrac{C}{T (\gamma-\sigma)}\left(\Psi(u^1)-\inf_{u\in\R^{d_1}\times\R^{d_2}}\Psi(u)+\left(\tfrac{3(\gamma-\sigma)}{\rho C}+\tfrac1{M\rho}\right)\E(\Upsilon_1)\right).
\end{align}
\end{Theorem}

\begin{proof}
By  \eqref{eq:upsilon}, Theorem \ref{thm:grad_decay} and \eqref{eq:exp_psi}
it holds for $0<\sigma<\gamma$ that
\begin{align}
\E\left(\Psi(u^k)-\Psi(u^{k+1})\right)
&\geq \gamma \E(\|u^{k+1}-u^k\|^2)+\tfrac1{M\rho}\E\left(\Upsilon_{k+1}-\Upsilon_k\right)\\
&\geq\sigma\E(\|u^{k+1}-u^k\|^2)+\tfrac{\gamma-\sigma}{C}\E\left(\dist(0,\partial F(x_1^{k+1},x_2^{k+1}))^2\right)\\&-\tfrac{3(\gamma-\sigma)}{C}\E(\Upsilon_k)+\tfrac1{M\rho}\E\left(\Upsilon_{k+1}-\Upsilon_k\right)\\
&\geq \sigma\E(\|u^{k+1}-u^k\|^2)+\tfrac{\gamma-\sigma}{C}\E\left(\dist(0,\partial F(x_1^{k+1},x_2^{k+1}))^2\right)\\&+\left(\tfrac{3(\gamma-\sigma)}{\rho C}+\tfrac1{M\rho}\right)\E\left(\Upsilon_{k+1}-\Upsilon_k\right)-\tfrac{3(\gamma-\sigma)V_\Upsilon}{C\rho}\E(\|u^{k+1}-u^k\|^2).
\end{align}
Choosing $\sigma \coloneqq \tfrac{3(\gamma-\sigma)V_\Upsilon}{C\rho}$ yields 
\begin{align}
\E\left(\Psi(u^k)\!-\!\Psi(u^{k+1})\right)
\geq
\tfrac{\gamma-\sigma}{C}\E\left(\dist(0,\partial F(x_1^{k+1},x_2^{k+1}))^2\right)
\!-\!\left(\tfrac{3(\gamma-\sigma)}{\rho C} \!+ \!\tfrac1{M\rho}\right)\E\left(\Upsilon_{k+1} \! -\! \Upsilon_k\right).
\end{align}
Adding this up for $k=1,...,T$ we get
\begin{align}
\E\left(\Psi(u^1)  -\Psi(u^{T})\right)
&\geq 
\tfrac{\gamma-\sigma}{C}\sum_{k=1}^{T}\E\left(\dist(0,\partial F(x_1^{k+1},x_2^{k+1}))^2\right) \\
& +  \left(\tfrac{3(\gamma-\sigma)}{\rho C}+\tfrac1{M\rho}\right)\E\left(\Upsilon_{T}-\Upsilon_1\right).
\end{align}
Since $\Upsilon_{T}\geq0$ this yields for $t$ drawn randomly from $\{2,...,T+1\}$ that
\begin{align}
\E\left(\dist(0,\partial F(x_1^{t},x_2^{t}))^2\right)
&=\tfrac1T\sum_{k=1}^{T}\E\left(\dist(0,\partial F(x_1^{k+1},x_2^{k+1}))^2\right)\\
&\leq\tfrac{C}{T(\gamma-\sigma)}\left(\Psi(u^1)-\inf_{u\in(\R^{d_1}\times\R^{d_2})^2}\Psi(u)+\left(\tfrac{3(\gamma-\sigma)}{\rho C}+\tfrac1{M\rho}\right)\E(\Upsilon_1)\right).
\end{align}
This finishes the proof.
\end{proof}

In \cite{DTLDS2020} the authors proved global convergence of the objective function 
evaluated at the iterates of SPRING in expectation if the global error bound
\begin{align}
F(x_1,x_2)-\ubar F\leq\mu\dist(0,\partial F(x_1,x_2))^2\label{eq:global_error_bound}
\end{align}
is fulfilled for some $\mu>0$.
Using this error bound, we can also prove global convergence of iSPALM in expectation with a linear convergence rate.
Note that the authors of \cite{DTLDS2020} used the generalized gradient instead of the subgradient also for this error bound. 
Similar as before this seems to be unsuitable due to the heavy dependence on of the generalized gradient on the step size parameters.

\begin{Theorem}[Convergence of iSPALM]\label{thm:ispring_conv_2}
Let the assumptions of Theorem \ref{thm:l2_steps} hold true. 
If in addition \eqref{eq:global_error_bound} is fulfilled, then there exists some 
$\Theta_0\in(0,1)$ and $\Theta_1>0$ such that
\begin{align}
\E\left(F(x_1^{T+1},x_2^{T+1})-\ubar F\right)\leq (\Theta_0)^T\left(\Psi(u^1)-\ubar F+\Theta_1\E(\Upsilon_1)\right).
\end{align}
In particular, it holds
$
\lim_{T\to\infty}\E(F(x_1^T,x_2^T)-\ubar F)=0.
$
\end{Theorem}

\begin{proof}
By \eqref{eq:exp_psi} and Theorem \ref{thm:grad_decay}, 
we obtain for $0<d<\min(\gamma,\tfrac{C\rho\mu}{1-\rho})$ that
\begin{align}
\E\left(\Psi(u^{k+1})-\ubar F+\tfrac1{M\rho}\Upsilon_{k+1}\right)
&\leq \E\left(\Psi(u^k)-\ubar F+\tfrac1{M\rho}\Upsilon_k\right)-\gamma\E(\|u^{k+1}-u^k\|^2)\\
&\leq \E\left(\Psi(u^k)-\ubar F+\tfrac1{M\rho}\Upsilon_k\right)\\
&-\tfrac{d}{C}\E\left(\dist(0,\partial F(x_1^{k+1},x_2^{k+1}))^2\right)+\tfrac{3d}{C}\E(\Upsilon_k)
\\&-\left(\gamma-d\right)\E(\|u^{k+1}-u^k\|^2).
\end{align}
Using  \eqref{eq:upsilon} in combination with the global error bound \eqref{eq:global_error_bound}, we get
\begin{align}
&\E\left(\Psi(u^{k+1})-\ubar F+\left(\tfrac{3d}{\rho C}+\tfrac1{M\rho}\right)\Upsilon_{k+1}\right)
\leq 
\E\left(\Psi(u^k)-\ubar F+\left(\tfrac{3d}{\rho C}+\tfrac1{M\rho}\right)\Upsilon_k\right)\\
&-\tfrac{d}{C\mu}\E\left(F(x_1^{k+1},x_2^{k+1})-\ubar F\right)
-\left(\gamma-d-\tfrac{3d V_\Upsilon}{\rho C}\right)\E(\|u^{k+1}-u^k\|^2).
\end{align}
Setting $C_\Upsilon \coloneqq \left(\tfrac{3d}{\rho C}+\tfrac1{M\rho}\right)$
and applying the definition \eqref{**} of $\Psi$, this implies 
\begin{align}
&\left(1+\tfrac{d}{C\mu}\right)\E\left(\Psi(u^{k+1})-\ubar F\right)
-
\tfrac{d}{C\mu}\E(\delta_1\Delta_1^{k+1}+\delta_2\Delta_2^{k+1})+C_\Upsilon\E(\Upsilon_{k+1})\\
&\leq
\E\left(\Psi(u^k)-\ubar F\right)+C_\Upsilon\E(\Upsilon_k)-\left(\gamma-d-\tfrac{3dV_\Upsilon}{\rho C}\right)\E(\|u^{k+1}-u^k\|^2).
\end{align}
With $\delta \coloneqq  \max (\delta_1,\delta_2)$ 
and $\Delta_1^{k+1}+\Delta_2^{k+1} \le \tfrac12\|u^{k+1}-u^k\|^2$ we get
\begin{align}
&\left(1+\tfrac{d}{C\mu}\right)\E\left(\Psi(u^{k+1})-\ubar F\right)+C_\Upsilon\E(\Upsilon_{k+1})\\
&\le 
\E\left(\Psi(u^k)-\ubar F\right)+C_\Upsilon\E(\Upsilon_k)-\left(\gamma-d-\tfrac{3dV_\Upsilon}{\rho C}
- \tfrac{d\delta}{2 C\mu}\right)\E(\|u^{k+1}-u^k\|^2).
\end{align}
Multiplying by $C_d \coloneqq \tfrac1{1+\tfrac{d}{C\mu}}=\tfrac{C\mu}{C\mu+d}$ this becomes
\begin{align}
&\E\left(\Psi(u^{k+1})-\ubar F\right)+ C_\Upsilon C_d \E(\Upsilon_{k+1})
\leq
\tfrac{C\mu}{C\mu+d}\E\left(\Psi(u^k)-\ubar F\right)+ C_\Upsilon C_d \E(\Upsilon_k)\\
&-\tfrac{C\mu}{C\mu+d}
\left(\gamma-d-\tfrac{3dV_\Upsilon}{\rho C}-\tfrac{d\delta}{ 2 C\mu}\right)\E(\|u^{k+1}-u^k\|^2).
\label{eq:mybound}
\end{align}                                                                     
Since $d<\tfrac{C\rho\mu}{1-\rho}$ we know that
$s\coloneqq \tfrac{1-C_d}{C_d+\rho-1} = \frac{d}{\rho C \mu + (\rho -1)d} > 0$. 
Thus, 
adding $s C_\Upsilon C_d$ times equation Definition \ref{def:ivr} (ii) to \eqref{eq:mybound} 
gives
\begin{align}
&\E\left( \Psi(u^{k+1})-\ubar F\right) + (1+s) C_\Upsilon C_d \E(\Upsilon_{k+1})
\leq C_d \E\left(\Psi(u^k)-\ubar F+ (1+s) C_\Upsilon C_d \E(\Upsilon_k)\right)\\
&+
C_d \underbrace{\left( V_\Upsilon s C_\Upsilon  - \left(\gamma-d-\tfrac{3dV_\Upsilon}{\rho C} -
\tfrac{d\delta}{2 C\mu}\right)\right)}_{\eqqcolon h(d)}\E(\|u^{k+1}-u^k\|^2),
\end{align}
where we have used that $1+(1-\rho)s=C_d(1+s)$.
Since $s$ converges to $0$ as $d\to0$ we have that $\lim_{d\to0}h(d)=-\gamma$. Thus we can choose $d>0$ small enough, such that $h(d)<0$. 
Then we get
\begin{align}
\E\left(\Psi(u^{k+1})-\ubar F\right)+ (1+s)C_\Upsilon C_d \E(\Upsilon_{k+1})
&\leq
C_d \E\left(\Psi(u^k)-\ubar F+ (1+s)C_\Upsilon C_d \E(\Upsilon_k)\right).
\end{align}
Finally, setting $\Theta_0 \coloneqq C_d$ and $\Theta_1 \coloneqq (1+s)C_\Upsilon C_d$ 
and applying the last equation iteratively, we obtain
\begin{align}
\E\left(\Psi(u^{T+1})-\ubar F+\Theta_1\Upsilon_{T+1}\right)&\leq(\Theta_0)^T\E\left(\Psi(u^1)-\ubar F+\Theta_1\Upsilon_1\right).
\end{align}
Note that $\Psi(u^{T+1})\geq F(x_1^{T+1},x_2^{T+1})$ and that $\Upsilon_{T+1}\geq0$. This yields
\begin{align}
\E\left(F(x_1^{T+1},x_2^{T+1})-\ubar F\right)
&\leq
(\Theta_0)^T\E\left(\Psi(u^1)-\ubar F+\Theta_1\Upsilon_1\right),
\end{align}
and we are done.
\end{proof}

%-------------------------------------------------------------
\section{Student-$t$ Mixture Models} \label{sec:student-t}
%-------------------------------------------------------------
In this section, we show how PALM and its inertial and stochastic variants 
can be applied to learn  Student-$t$ MMs. 
To this end, we denote 
by $\mathrm{Sym}(d)$ the linear space of symmetric $d \times d$ matrices, 
by 
$\SPD(d)$ the cone of symmetric, positive definite $d \times d$ matrices
and by
$\Delta_K \coloneqq \{\alpha = (\alpha_k)_{k=1}^K: \sum_{k=1}^K \alpha_k = 1, \; \alpha_k \ge 0\}$
the probability simplex in $\R^K$.
The density function of the 
$d$-dimensional Student-$t$ distribution $T_\nu(\mu,\Sigma)$ with 
$\nu>0$ degrees of freedom, \emph{location} parameter $\mu\in \R^d$ and \emph{scatter matrix} $\Sigma\in \SPD(d)$ 
is given by
\begin{equation}\label{pdf}
f(x|\nu,\mu,\Sigma)  = 
\frac{\Gamma\left(\frac{d+\nu}{2}\right)}{\Gamma\left(\frac{\nu}{2}\right)\, \nu^{\frac{d}{2}} \, \pi^{\frac{d}{2}} \,
{\abs{\Sigma}}^{\frac{1}{2}}} \, \frac{1}{\left(1 +\frac1\nu(x-\mu)^\tT \Sigma^{-1}(x-\mu) \right)^{\frac{d+\nu}{2}}},
\end{equation}
with the \emph{Gamma function}
$
\Gamma(s) \coloneqq\int_0^\infty t^{s-1}\e^{-t}\dx[t] 
$.
The expectation of the Student-$t$ distribution is $\E(X) = \mu$ for $\nu > 1$ 
and the covariance matrix is given by $\Cov(X) =\frac{\nu }{\nu-2} \Sigma$ for $\nu > 2$, 
otherwise these quantities are undefined. 
The smaller the value of $\nu$, the heavier are the tails of the $T_\nu(\mu,\Sigma)$ distribution.
For $\nu \to \infty$, 
the Student-$t$ distribution $T_\nu(\mu,\Sigma)$ converges to the normal distribution $\NN(\mu,\Sigma)$ and for $\nu = 0$
it is related to the projected normal distribution on the sphere $\SP^{d-1}\subset\R^d$.
Figure~\ref{Fig:different_nu} 
illustrates this behavior for the one-dimensional standard Student-$t$ distribution. 

\begin{figure}[thb]
\centering  
\centering  
{\includegraphics[width=0.4\textwidth]{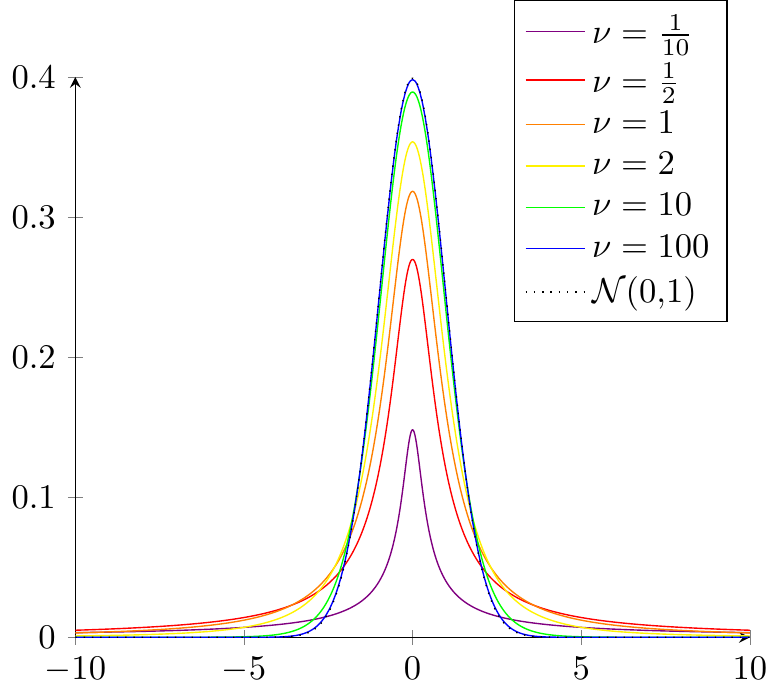}}
\caption{Standard Student-$t$ distribution $T_\nu(0,1)$ 
for different values of $\nu$ in comparison with the standard normal distribution $\NN(0,1)$.}\label{Fig:different_nu}
\end{figure}

The construction of MMs arises from the following scenario: 
we have $K$ random number generators sampling from different distributions. Now we first choose one of the random number generators randomly using the probability weights $\alpha=(\alpha_1,...,\alpha_K) \in\Delta_K$ and sample from the corresponding distribution. 
If all random number generators sample from Student-$t$ distributions we arrive at Student-$t$ MMs.
More precisely, if $Y$ is a random variable mapping into $\{1,...,K\}$ and $X_1,...,X_k$ are random variables with 
$X_k\sim T_{\nu_k}(\mu_k,\Sigma_k)$, 
then the random variable  $X_Y$ is a Student-$t$ MM with probability density function
\begin{equation}\label{eq:density_mm}
p(x)=\sum_{k=1}^K \alpha_k f(x|\nu_k,\mu_k,\Sigma_k), \qquad \alpha\in\Delta_K.
\end{equation}
For samples $\mathcal X=(x_1,...,x_n)$, we aim to find the parameters of the the Student-$t$ MM
by minimizing its negative log-likelihood function 
\begin{equation}\label{eq:logLike}
\mathcal L(\alpha,\nu,\mu,\Sigma|\mathcal X)=-\frac1n\sum_{i=1}^n\log \bigg(\sum_{k=1}^K\alpha_k f(x_i|\nu_k,\mu_k,\Sigma_k)\bigg)
\end{equation}
subject to the parameter constraints.
A first idea to rewrite this problem in the form \eqref{eq:PALM_min_general} looks as
\begin{align}
F(\alpha,\nu,\mu,\Sigma)=H(\alpha,\nu,\mu,\Sigma)+f_1(\alpha)+f_2(\nu)+f_3(\mu)+f_4(\Sigma),\label{eq:obj_student_t}
\end{align}
where $H \coloneqq \mathcal{L}$, $f_1 \coloneqq \iota_{\Delta_K}$, $f_2 \coloneqq \iota_{\R^K_{>0}}$, $f_3 \coloneqq 0$, 
$f_4 \coloneqq \iota_{\SPD(d)^K}$, and 
$\iota_{\mathcal S}$ denotes the \emph{indicator function} of the set ${\mathcal S}$
defined by $\iota_{\mathcal S}(x) \coloneqq 0$ if $x \in \mathcal S$ and $\iota_{\mathcal S}(x) \coloneqq \infty$ otherwise.
Indeed one of the authors has applied PALM and iPALM to such a setting
without any convergence guarantee in \cite{Hertrich2020}. 
The problem is that $\mathcal{L}$ is not defined on the whole Euclidean space 
and since $\mathcal{L}(\alpha,\nu,\mu,\Sigma)\to\infty$ if $\Sigma_k\to 0$ 
for some $k$, the function can also not continuously extended to the whole $\R^K\times\R^K\times\R^{d\times K}\times\Sym(d)^K$. 
Furthermore, the functions $f_2$ and $f_4$ are not lower semi-continuous. 
Consequently, the function \eqref{eq:obj_student_t} does not fulfill 
the assumptions required for the convergence of PALM and iPALM as well as their stochastic variants.
Therefore we modify the above model as follows:
Let  $\SPD_\epsilon(d) \coloneqq \{\Sigma \in \SPD(d): \Sigma \succeq\epsilon I_d\}$.
Then we use the surjective mappings 
$\varphi_1 \colon \R^K \to\Delta_K$, 
$\varphi_2 \colon \R^K\to\R_{\geq\epsilon}^K$ and 
$\varphi_3  \colon  \mathrm{Sym}(d)^K \to \SPD_\varepsilon(d)^K$
defined by 
\begin{align} \label{trafo}
\varphi_1(\alpha) \coloneqq \frac{\exp(\alpha)}{\sum_{j=1}^K\exp(\alpha_j)},\quad
\varphi_2 (\nu) \coloneqq \nu^2+\epsilon ,\quad 
\varphi_3(\Sigma) \coloneqq \left( \Sigma_k^T\Sigma_k+\epsilon I_d \right)_{k=1}^K 
\end{align}
to reshape problem \eqref{eq:obj_student_t} as the unconstrained optimization problem 
\begin{align} \label{eq:unconstrainted_problem}
\argmin_{\alpha\in\R^K,\nu\in\R^K,\mu\in\R^{d\times K},\Sigma\in\Sym(d)^{K}} H(\alpha, \nu,\mu,\Sigma)
\coloneqq
{\mathcal L}(\varphi_1(\alpha),\varphi_2(\nu),\mu,\varphi_3 (\Sigma)|\mathcal X).
\end{align}
Note that the functions $f_i$, $i=1,\ldots,4$ are just zero.

%------------------------------------------------------------------------------------------------------
For problem \eqref{eq:unconstrainted_problem}, PALM and iPALM
reduce basically to  block gradient descent algorithms as in
Algorithm \ref{alg:PALM_mm} and \ref{alg:iPALM_mm}, respectively.
Note that we use $\beta_i^k = 0$ for all $k,i$ and $\alpha_i^k = \rho^k$ for $i=1,\ldots,4$ 
as iPALM parameters in Algorithm \ref{alg:iPALM_mm}.
For the stochastic variants SPRING and iSPALM, 
we have just to replace the gradient by a stochastic gradient estimator.

\begin{algorithm}[!ht]
\caption{Proximal Alternating Linearized Minimization (PALM) for Student-$t$ MMs}\label{alg:PALM_mm}
\begin{algorithmic}
\State \textbf{Input:} $x_1,\ldots,x_n\in\R^d$, $\alpha^{0}\in\R^K$, $\nu^{0}\in\R^K$, $\mu^{0}\in \R^{d\times K}$, 
$\Sigma^{0}\in\R^{d\times d\times K}$, $\tau_1^k,\tau_2^k,\tau_3^k,\tau_4^k$ for $k\in\N$
\For {$k=1,...$}
\State \textbf{$\alpha$-Update:}
\State $\alpha^{k+1}=\alpha^{k}-\frac1{\tau_1^k}
\nabla_{\alpha} H(\alpha^{k},\nu^{k},\mu^{k},\Sigma^{k})$\\
\State \textbf{$\nu$-Update:}
\State $\nu^{k+1}=\nu^{k}-\frac1{\tau_2^k}\nabla_{\nu} H(\alpha^{k+1},\nu^{k},\mu^{k},\Sigma^{k})$\\
\State \textbf{$\mu$-Update:}
\State $\mu^{k+1}=\mu^{k}-\frac1{\tau_3^k}\nabla_{\mu} H(\alpha^{k+1},\nu^{k+1},\mu^{k},\Sigma^{k})$\\
\State \textbf{$\Sigma$-Update:}
\State $\Sigma^{k+1}=\Sigma^{k}-\frac1{\tau_4^k}\nabla_{\Sigma}
H(\alpha^{k+1},\nu^{k},\mu^{k+1},\Sigma^{k})$
\EndFor
\end{algorithmic}
\end{algorithm}

\begin{algorithm}[!ht]
\caption{Inertial Proximal Alternating Linearized Minimization (iPALM) for Student-$t$ MMs}\label{alg:iPALM_mm}
\begin{algorithmic}
\State \textbf{Input:} $x_1,\ldots,x_n\in\R^d$, $\alpha^{0}\in\R^K$, $\nu^{0}\in\R^K$, $\mu^{0}\in \R^{d\times K}$, 
$\Sigma^{0}\in\R^{d\times d\times K}$, $\rho^k\in[0,1]$, $\tau_1^k,\tau_2^k,\tau_3^k,\tau_4^k$ for $k\in\N$
\For {$k=1,...$}
\State \textbf{$\alpha$-Update:}
\State $\alpha_z^k=\alpha^k+\rho^k(\alpha^k-\alpha^{k-1})$
\State $\alpha^{k+1}=\alpha_z^k-\frac1{\tau_1^k}\nabla_{\alpha}H(\alpha_z^k,\nu^k,\mu^k,\Sigma^k)$\\
\State \textbf{$\nu$-Update:}
\State $\nu_z^k=\nu^k+\rho^k(\nu^k-\nu^{k-1})$
\State $\nu^{k+1}=\nu_z^k-\frac1{\tau_2^k}\nabla_{\tilde\nu}H(\alpha^{k+1},\nu_z^k,\mu^k,\Sigma^k)$\\
\State \textbf{$\mu$-Update:}
\State $\mu_z^k=\mu^k+\rho^k(\mu^k-\mu^{k-1})$
\State $\mu^{k+1}=\mu_z^k-\frac1{\tau_3^k}\nabla_{\mu}H(\alpha^{k+1},\nu^{k+1},\mu_z^k,\Sigma^k)$\\
\State \textbf{$\Sigma$-Update:}
\State $\Sigma_z^k=\Sigma^k+\rho^k(\Sigma^k-\Sigma^{k-1})$
\State $\Sigma^{k+1}=\Sigma_z^k-\frac1{\tau_4^k}\nabla_{\Sigma} H(\alpha^{k+1},\nu^{k+1},\mu^{k+1},\Sigma_z^k)$
\EndFor
\end{algorithmic}
\end{algorithm}
%-----------------------------------------------------------------------------------

Finally, we will show that $H$ in \eqref{eq:unconstrainted_problem}
\begin{itemize}
\item is a KL function which is bounded from below, and
\item satisfies the Assumption \ref{ass:PALM2}(i).
\end{itemize} 
Since $H \in C^2(\R^K\times \R^K\times \R^{d\times K} \times\Sym(d)^{K})$ 
we know by Remark \ref{rem:PALM_other_ass} that Assumption \ref{ass:PALM2}(ii)
is also fulfilled. Further, $\nabla H$ is continuous on bounded sets.
Then, choosing the parameters of PALM, resp. iPALM as required by Theorem \ref{thm:PALM_convergence}
resp. \ref{thm:iPALM_convergence}, we conclude that the sequences generated by both algorithms converge to a critical point of $H$
supposed that they are bounded.
Similarly, if we assume in addition that the stochastic gradient estimators are variance-reduced, resp.
inertial variance-reduced, we can conclude that the sequences of SPRING and iSPALM converge as in Theorem \ref{thm:conv_spring}
resp. Theorems \ref{thm:ispring_conv_1} and \ref{thm:ispring_conv_2}, 
if the corresponding requirements on the parameters are fulfilled.
\\

We start with the KL property.

\begin{Lemma}\label{lem:KL_Likelihood}
The function $H: \R^K \times \R^K \times \R^{d\times K} \times \Sym(d)^{K}\to \R$ defined in \eqref {eq:unconstrainted_problem}
is a KL function. Moreover, it is bounded from below.
\end{Lemma}

\begin{proof}
1. Since the Gamma function is real analytic, we have that $H$ is a combination of sums, products, 
quotients and concatenations of real analytic functions. Thus $H$ is real analytic. 
This implies that it is a KL function, see \cite[Example 1]{AB2009} and \cite{L1963,L1993}.

2. 
First, we proof that $f(x|\nu,\mu,\Sigma)$ is bounded from above for $\nu>\epsilon, \mu\in\R^d$ and $\Sigma\succeq \epsilon I_d$.
By definition of the Gamma function and since
\begin{align}
\Gamma(\tfrac{\nu+d}{2})/\Gamma(\tfrac\nu2)\nu^{\tfrac{d}2} \to1\quad\text{as}\quad \nu\to\infty\label{eq:gamm_to_infty}
\end{align}
we have that \eqref{eq:gamm_to_infty} is bounded from below for $\nu\in[\epsilon,\infty)$. 
Further, we see by assumptions on $\nu$ and $\Sigma$ that 
$
\abs{\Sigma}^{-\tfrac12} \leq \epsilon^{-\tfrac{d}{2}}$
and
$\left( 1+ \tfrac1\nu (x-\mu)^\tT \Sigma^{-1} (x - \mu) \right)^{-\tfrac{d+\nu}{2}} \leq 1$.
Thus, $f(x|\nu,\mu,\Sigma)$ is the product of bounded functions and therefore itself bounded
by some $C>0$. 
This yields for 
$\tilde \alpha=\varphi_1(\alpha)$, $\tilde \nu = \varphi_2 (\nu)$ 
and $\tilde \Sigma = \varphi_3 (\Sigma)$ that 
$$
-\sum_{i=1}^n\log\Big(\sum_{k=1}^K \tilde \alpha_k f(x_i| \tilde \nu_k,\tilde \mu_k,\tilde \Sigma_k)\Big)
\leq-\sum_{i=1}^n\log\Big(\sum_{k=1}^K \tilde \alpha_k C\Big) \leq - n\log C,
$$
which finishes the proof.
\end{proof}

Here are the Lipschitz properties of $H$.

\begin{Lemma}\label{lem:Lipschitz_Likelihood}
For $H: \R^K \times \R^K \times \R^{d\times K} \times \Sym(d)^{K}\to \R$ defined by \eqref {eq:unconstrainted_problem} 
and all
$\alpha\in\R^K$,
$\nu\in\R^K$, 
$\mu\in\R^{d\times K}$ and $\Sigma\in\R^{d\times d\times K}$
we have that the gradients
$\nabla_{\alpha} H(\cdot,\nu,\mu,\Sigma)$,
$\nabla_{\nu} H (\alpha,\cdot,\mu,\Sigma)$,
$\nabla_\mu H(\alpha,\nu,\cdot,\Sigma)$, and
$\nabla_{\Sigma}H(\alpha,\nu,\mu,\cdot)$ 
are globally Lipschitz continuous.
\end{Lemma}

The technical proof of the lemma is given in Appendix \ref{sec:proof}.

%-----------------------------------------------
\section{Numerical Results} \label{sec:numerics}
%-----------------------------------------------

In this section, we apply iSPALM with SARAH gradient estimator (iSPALM-SARAH) 
for two different applications and compare it with PALM, iPALM and SPRING-SARAH. 
We found numerically, that it increases the stability of SPRING-SARAH and iSPALM-SARAH if we enforce the evaluation of the full gradient at the beginning of each epoch.
We run all our experiments on a Lenovo ThinkStation with Intel i7-8700 processor, 32GB RAM and a NVIDIA GeForce GTX 2060 Super GPU. 
For the implementation we use Python and Tensorflow.

\subsection{Parameter Choice and Implementation Aspects}

On the one hand, the algorithms based on PALM have many parameters
which enables a high adaptivity of the algorithms to the specific problems. 
On the other hand, it is often hard to fit these parameters 
to ensure the optimal performance of the algorithms. 

Based on approximations $\tilde L_1(x_2^k)$ and $\tilde L_2(x_1^{k+1})$
of the partial Lipschitz constants $L_1(x_2^k)$ and $L_2(x_1^{k+1})$
outlined below, we use the following step size parameters $\tau_i^k$, $i=1,2$:
\begin{itemize}
\item For \textbf{PALM} and \textbf{iPALM}, 
we choose $\tau_1^k=\tilde L_1(x_1^k,x_2^k)$ and $\tau_2^k=\tilde L_2(x_1^{k+1},x_2^k)$
which was also suggested in \cite{BST2014,PS2016}.
\item For \textbf{SPRING-SARAH} and \textbf{iSPALM-SARAH}, 
we choose $\tau_1^k=s_1\tilde L_1(x_1^k,x_2^k)$ and $\tau_2^k=s_1\tilde L_2(x_1^{k+1},x_2^k)$, 
where the manually chosen scalar $s_1 > 0$ depends on the application.
Note that the authors in \cite{DTLDS2020} propose to take $s_1=2$ which was not optimal
in our examples.
\end{itemize}

\paragraph{Computation of Gradients and Approximative Lipschitz Constants}

Since the global and partial Lipschitz constants of $H$ are usually unknown, 
we estimate them locally using the second order derivative of $H$ which exists in our examples.
If $H$ acts on a high dimensional space, it is often computationally to
costly to compute the full Hessian matrix. 
Thus we compute a local Lipschitz constant only in the gradient direction, 
i.e.\ we compute 
\begin{equation}\label{approx_Lip}
\tilde L_i(x_1,x_2) \coloneqq \|\nabla_{x_i}^2 H(x_1,x_2)g\|, \quad
g\coloneqq \frac{\nabla_{x_i} H(x_1,x_2)}{\|\nabla_{x_i}H(x_1,x_2)\|}
\end{equation}
For the stochastic algorithms we replace $H$ by the approximated function $\tilde H(x_1,x_2)\coloneqq\tfrac1b\sum_{i\in B_i^k}h_i(x_1,x_2)$, 
where $B_i^k$ is the current mini-batch.
The analytical computation of $\tilde L_i$ in \eqref{approx_Lip} is still hard. 
Even computing the gradient of a complicated function $H$ can be error prone and laborious. 
Therefore, we compute the (partial) gradients of $H$ or $\tilde H$, respectively, 
using the reverse mode of algorithmic differentiation (also called backpropagation), see e.g.\ \cite{GW2008}.
To this end, note that the chain rule yields that
\begin{align}
\left\|\nabla_{x_i}\left(\|\nabla_{x_i}H(x_1,x_2)\|^2\right)\right\|&=\left\|2\|\nabla_{x_i}H(x_1,x_2)\| \nabla_{x_i}^2 H(x_1,x_2)\nabla_{x_i}H(x_1,x_2)\right\|\\
&=2\|\nabla_{x_i}H(x_1,x_2)\|^2\tilde L_i(x_1,x_2).
\end{align}
Thus, we can compute $\tilde L_i(x_1,x_2)$ by applying two times the reverse mode.
If we neglect the taping, the execution time of this procedure can provably be bounded 
by a constant times the execution time of $H$, see \cite[Section 5.4]{GW2008}. 
Therefore, this procedure gives us an accurate and computationally very efficient estimation of the local partial Lipschitz constant.

%--------------------------------------------------------------
\paragraph{Inertial Parameters}

For the iPALM and iSPALM-SARAH we have to choose the inertial parameters $\alpha_i^k\geq0$ and $\beta_i^k\geq0$. 
With respect to our convergence results we have to assume that there exist 
$\alpha_i^k\leq \bar\alpha_i<\tfrac12$ and $\beta_i^k\leq\bar\beta_i<1$, $i=1,2$. 
Note that for convex functions $f$ and $g$, the authors in \cite{PS2016} proved  that
the assumption on the $\alpha$'s can be lowered to $\alpha_i^k\leq \bar\alpha_i<1$ and suggested to use
$\alpha_i^k=\beta_i^k=\frac{k-1}{k+2}$.
Unfortunately, we cannot show this for iSPALM
and indeed we observe instability and divergence in iSPALM-SARAH, 
if we choose $\alpha_i^k>\frac12$.
Therefore, we choose for iSPALM-SARAH the parameters
\begin{align}
\alpha_i^k=\beta_i^k= s_2\frac{k-1}{k+2},
\end{align}
where the scalar $0<s_2<1$ is manually chosen depending on the application.

%--------------------------------------------------------------
%\paragraph{Initialization} 
%We observed that SPRING-SARAH and iSPALM-SARAH show a slow convergence behavior for a poor initializations. 
%Thus, we use a so-called warm start for the algorithms, 
%similarly as in \cite{KR2017} and \cite{DTLDS2020}. 
%More precisely, we pre-process our initialization by performing two steps of PALM before comparing the algorithms.

%--------------------------------------------------------------
\paragraph{Implementation}
We provide a general framework for implementing PALM, iPALM, SPRING-SARAH and
iSPALM-SARAH\footnote{\url{https://github.com/johertrich/Inertial-Stochastic-PALM}} on a GPU. Using this framework, it suffice to
provide an implementation for the functions $H$ and $\prox_{\tau_i}^{f_i}$ in order to use one of the
above algorithms for the function $F(x_1,...,x_K)=H(x_1,...,x_K)+\sum_{i=1}^K f_i(x_i)$. We provide also the code of our numerical examples below on this website.

\subsection{Student-$t$ Mixture Models}
We  estimate the parameters of the Student-$t$ MM \eqref{eq:unconstrainted_problem}
with $K$ components and data points $\mathcal X =(x_1,...,x_n)\in\R^{d\times n}$. 
We generate the data by sampling from a Student-$t$ MM as described above. 
The parameters of the ground truth MM are generate as follows:
\begin{itemize}
\item We generate $\alpha=\frac{\bar\alpha^2+1}{\|\bar\alpha^2+1\|_1}$, where the entries of $\bar\alpha\in\R^K$ are drawn independently from the standard normal distribution.
\item We generate $\nu_i=\min(\bar\nu_i^2+1,100)$, where $\bar\nu_i$, $i=1,\ldots,n$ is drawn from a normal distribution with mean $0$ 
and standard deviation $10$.
\item The entries of $\mu\in\R^{d\times K}$ are drawn independently from a normal distribution with mean $0$ and standard deviation $2$.
\item We generate 
$\Sigma_i=\bar\Sigma_i^T\bar\Sigma_i + I$, where the entries of $\bar\Sigma_i\in\R^{d\times d}$ 
are drawn independently from the standard normal distribution.
\end{itemize}
For the initialization of the algorithms, we assign to each sample $x_i$ randomly a class $c_i\in\{1,...,K\}$. 
Then we initialize the parameters $(\nu_k,\mu_k,\Sigma_k)$ by estimating the parameters of a Student-t distribution of all samples with $c_i=k$ using a faster alternative of the EM algorithm called multivariate myriad filter, see \cite{HHLS2019}.
Further we initialize $\alpha$ by $\alpha_k=\frac{|\{i\in\{1,...,N\}:c_i=k\}|}{N}$.
We run the algorithm for $n=200000$ data points of dimension $d=10$ and $K=30$ components. We use a batch size of $b=20000$.
To represent the randomness in SPRING-SARAH and iSPALM-SARAH, we repeat the experiment $100$ times with the same samples and the same initialization.
The resulting mean and standard deviation of the negative log-likelihood values versus the number of epochs and the execution times, respectively, 
are given in Figure \ref{fig:results_student_t}.
Further, we visualize the mean squared norm of the gradient after each epoch.
One epoch contains for SPRING-SARAH and iSPALM-SARAH $10$ steps and for PALM and iPALM $1$ step. 
We see that in terms of the number of epochs as well as in terms of the execution time the iSPALM-SARAH is the fastest algorithm.

\begin{figure}[t]
\begin{subfigure}[t]{0.5\textwidth}
\centering
\includegraphics[width=\textwidth]{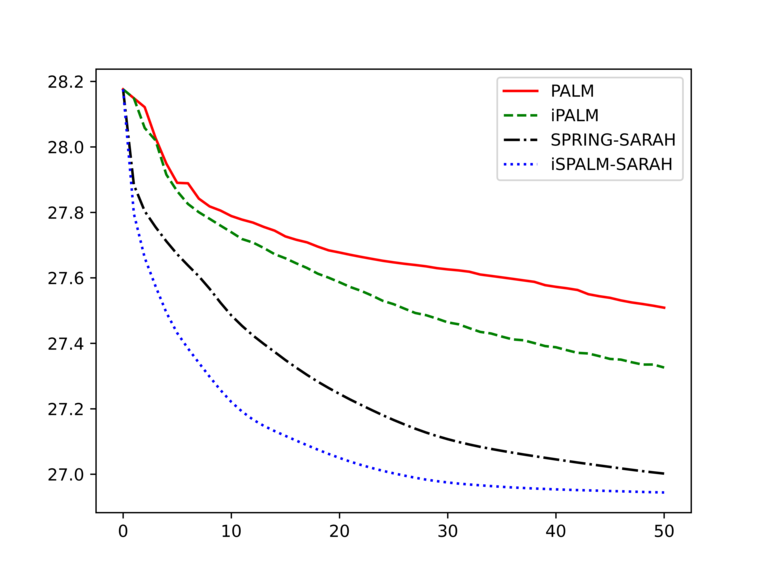}
\caption{Average objective versus epochs}
\end{subfigure}\hfill
\begin{subfigure}[t]{0.5\textwidth}
\centering
\includegraphics[width=\textwidth]{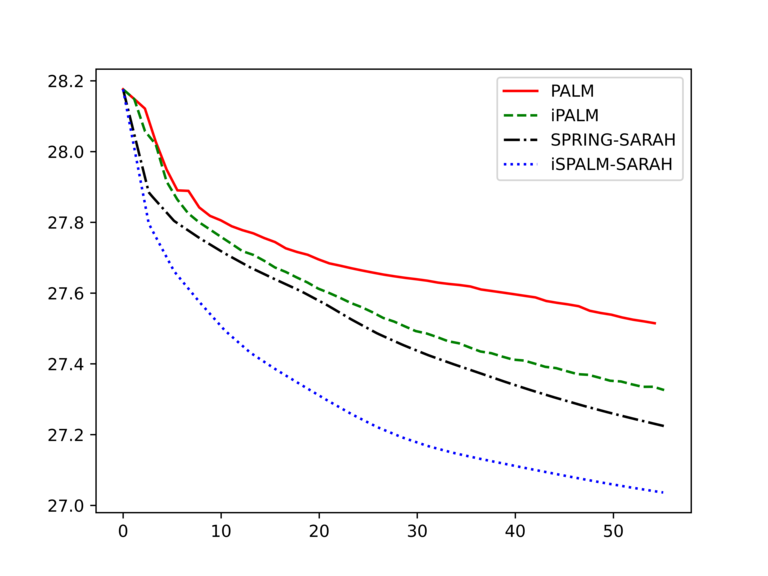}
\caption{Average objective versus execution time}
\end{subfigure}
\begin{subfigure}[t]{0.5\textwidth}
\centering
\includegraphics[width=\textwidth]{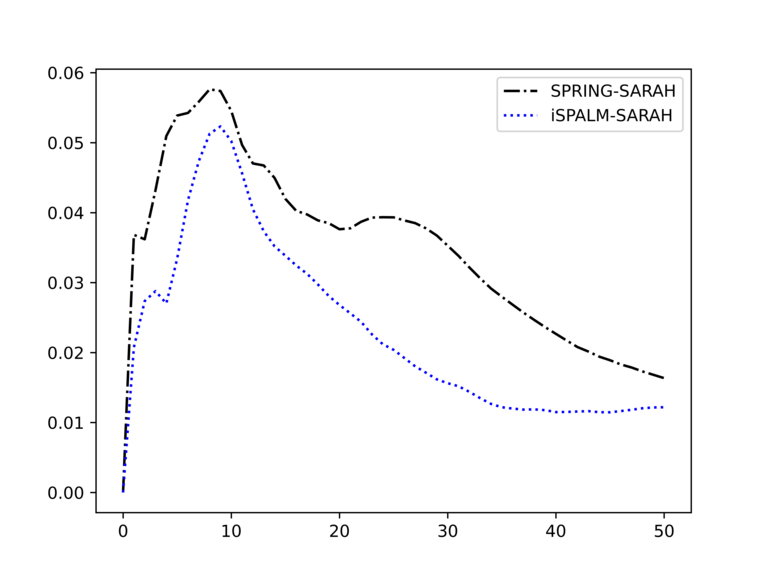}
\caption{Standard deviation of the objective versus epochs}
\end{subfigure}
\begin{subfigure}[t]{0.5\textwidth}
\centering
\includegraphics[width=\textwidth]{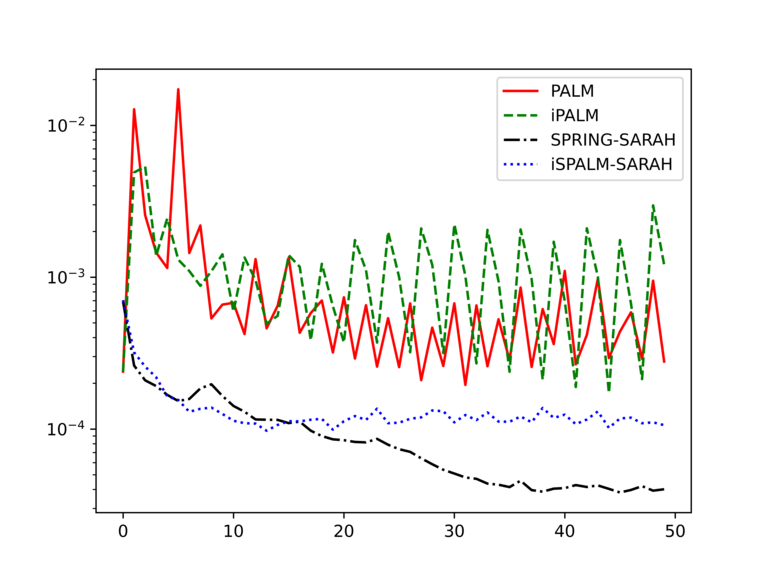}
\caption{Average squared norm of the gradient versus epochs}
\end{subfigure}
\caption{Objective function versus number of epochs and versus execution time for estimating the parameters of Student-$t$ MMs}
\label{fig:results_student_t}
\end{figure}

%---------------------------------------------
\subsection{Proximal Neural Networks (PNNs)}
%---------------------------------------------

\paragraph{PNNs for MNIST classification}
In this example, we train a Proximal Neural Network as introduced in \cite{HHNPSS2020} for classification on the MNIST data set\footnote{\url{http://yann.lecun.com/exdb/mnist}}. The training data consists of $N=60000$ images $x_i\in\R^d$ of size $d=28^2$ and labels $y_i\in\{0,1\}^{10}$, where the $j$th entry of $y_i$ is $1$ if and only if $x_i$ has the label $j$.
A PNN with $K-1$ layers and activation function $\sigma$ is defined by
$$
T_{K-1}^\tT\sigma(T_{K-1} ... T_1^\tT\sigma(T_1x+b_1) ... +b_{K-1}),
$$
where the $T_i$ are contained in the (compact) Stiefel manifold $\St(d,n_i)$ and $b_i\in\R^{n_i}$ for $i=1,...,K-1$.
To get $10$ output elements in $(0,1)$, we add similar as in \cite{HHNPSS2020} an additional layer
$$
g(T_Kx),\quad T_K\in[-10,10]^{10,d}, b_K\in\R^{10}
$$
with the activation function $g(x)\coloneqq\tfrac{1}{1+\exp(-x)}$.
Thus the full network is given by
$$
\Psi(x,u)=g(T_KT_{K-1}^\tT\sigma(T_{K-1}...T_1^\tT\sigma(T_1x+b_1)+...+b_{K-1}) +b_K),\quad u=(T_1,...,T_K,b_1,...,b_K).
$$
It was demonstrated in \cite{HHNPSS2020} that this kind of network is more stable under adversarial attacks than the same network without the orthogonality constraints.

\paragraph{Training PNNs with iSPALM-SARAH}

Now, we want to train a PNN with $K-1=3$ layers and $n_1=784$, $n_2=400$ and $n_3=200$ for MNIST classification. In order of applying our theory, we use the exponential linear unit (ELU)
$$
\sigma(x)=\begin{cases}\exp(x)-1,&$if $x<0,\\x&$if $x\ge 0,\end{cases}
$$
as activation function, which is differentiable with a $1$-Lipschitz gradient.
Then, the loss function is given by
$$
F(u)=H(u)+f(u),\quad u=(T_1,...,T_4,b_1,...,b_4)
$$
where $T_i \in \R^{d,n_i}$, $b_i \in \mathbb R^{n_i}$, $i=1,2,3$, and
$T_4 \in [-10,10]^{10,d}$, $b_4 \in \R^{10}$,
and $f(u)=\iota_{\mathcal U}$ with
$$
\mathcal U\coloneqq\{(T_1,...,T_4,b_1,...,b_4):T_i\in\St(d,n_i), i=1,2,3, T_4\in[-10,10]^{10,d}\}.
$$
and
$$
H(u)\coloneqq\frac1N\sum_{i=1}^N \|\Psi(x_i,u) - y_i\|^2.
$$

Since $H$ is unfortunately  not Lipschitz continuous, we propose a slight modification. 
Note that for any $u=(T_1,...,T_4,b_1,...,b_4)$ which appears as $x^k$, $y^k$ or $z^k$ in PALM, iPALM, SPRING-SARAH or iSPALM-SARAH we have that there exist $v,w\in \mathcal U$ such that $u=v+w$. 
In particular, we have that $\|T_i\|_F\leq 2\sqrt{d}$, $i=1,2,3$ and $\|T_4\|_F\leq 20\sqrt{10 d}$
Therefore, we can replace $H$ by
$$
\tilde H(u)=\Pi_{i=1}^4\eta(\|T_i\|_F^2)\frac1N\sum_{i=1}^N \| \Psi(x_i,u)- y_i\|^2,
$$
without changing the algorithm, where $\eta$ is a smooth cutoff function of the interval $(-\infty,4000 d]$. 
Now, simple calculations yield that the function $\tilde H$ is globally Lipschitz continuous.
Since it is also bounded from below by $0$ we can conclude that our convergence results of iSPALM-SARAH are applicable.

\begin{Remark}
For the implementation, we need to calculate $\prox_{\tilde f}$, which is the orthogonal projection $P_{\mathcal U}$ onto $\mathcal U$. 
This includes the projection of the matrices $T_i$, $i=1,2,3$ onto the Stiefel manifold. 
In \cite[Section 7.3,7.4]{HS2013} it is shown, that the projection of a matrix $A$ onto the Stiefel manifold is given by the $U$-factor of the polar decomposition $A=US\in\R^{d,n}$, where $U\in\St(d,n)$ and $S$ is symmetric and positive definite.
Note that $U$ is only unique, if $A$ is non-singular.
Several possibilities for the computing $U$ are considered in \cite[Chapter 8]{Higham08}.
In particular, $U$ is given by $VW$, where $A=V\Sigma W$ is the singular value decomposition of $A$. 
For our numerical experiments we use the iteration
$$
Y_{k+1}=2Y_k(I+Y_k^\tT Y_k)^{-1}
$$
with $Y_0=A$, which converges for any non-singular $A$ to $U$, see \cite{Higham08}.
\hfill $\Box$
\end{Remark}

Now we run PALM, iPALM, SPRING-SARAH and iSPRING-SARAH algorithms for $200$ epochs using a batch size of $b=1500$. One epoch contains for SPRING-SARAH and iSPALM-SARAH $40$ steps and for PALM and iPALM $1$ step. 
We repeat the experiment 10 times with the same initialization and plot for the resulting loss functions mean and standard deviation to represent the randomness of the algorihtms.
%We compare our results with the stochastic gradient descent on the Stiefel manifold as proposed in \cite{HHNPSS2020}.
Figure \ref{fig:results_PNNs} shows the mean and standard deviation of the loss versus the number of epochs or the execution time 
as well as the squared norm of the Riemannian gradient for the iterates of iSPALM-SARAH after each epoch.
We observe that iSPALM-SARAH performs much better than SPRING-SARAH and that iPALM performs much better than PALM.
Therefore this example demonstrates the importance of the inertial parameters in iPALM and iSPALM-SARAH.
Further, iSPALM-SARAH and SPRING-SARAH outperform their deterministic versions significantly.
The resulting weights from iSPALM-SARAH reach after $200$ epochs an average accuracy of $0.985$ on the test set.

\begin{figure}[!ht]
\begin{subfigure}[t]{0.5\textwidth}
\centering
\includegraphics[width=\textwidth]{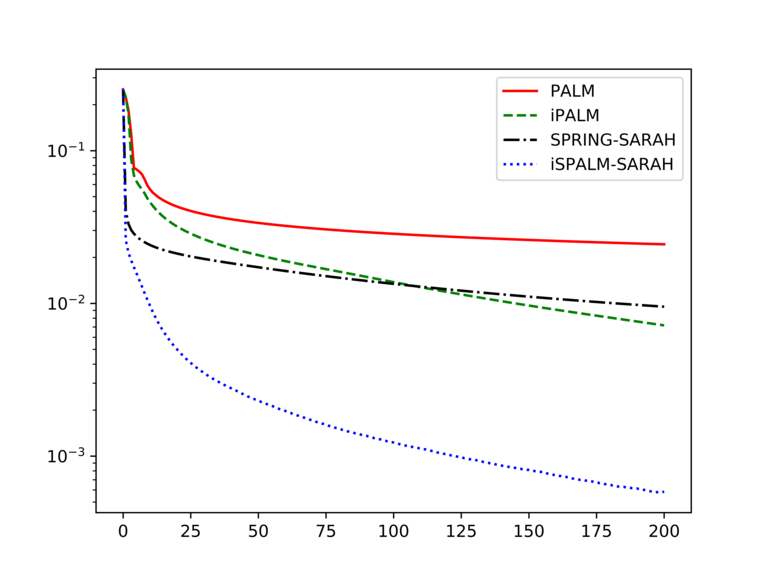}
\caption{Average loss versus epochs on the training set.}
\end{subfigure}\hfill
\begin{subfigure}[t]{0.5\textwidth}
\centering
\includegraphics[width=\textwidth]{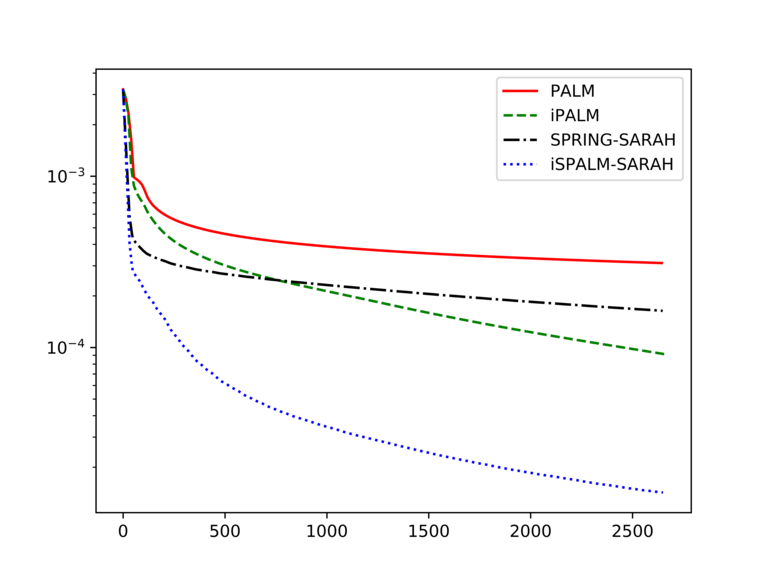}
\caption{Average loss versus execution time on the training set.}
\end{subfigure}
\begin{subfigure}[t]{0.5\textwidth}
\centering
\includegraphics[width=\textwidth]{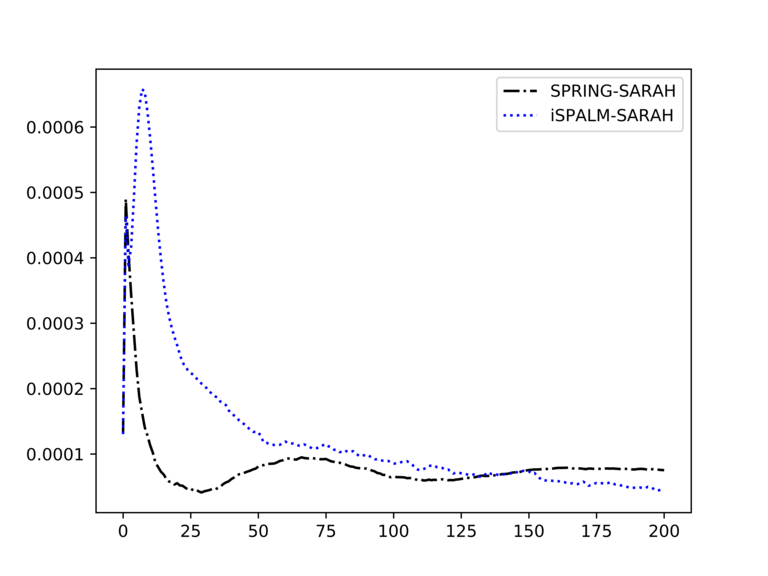}
\caption{Standard deviation of the loss versus execution time on the training set.}
\end{subfigure}\hfill
\begin{subfigure}[t]{0.5\textwidth}
\centering
\includegraphics[width=\textwidth]{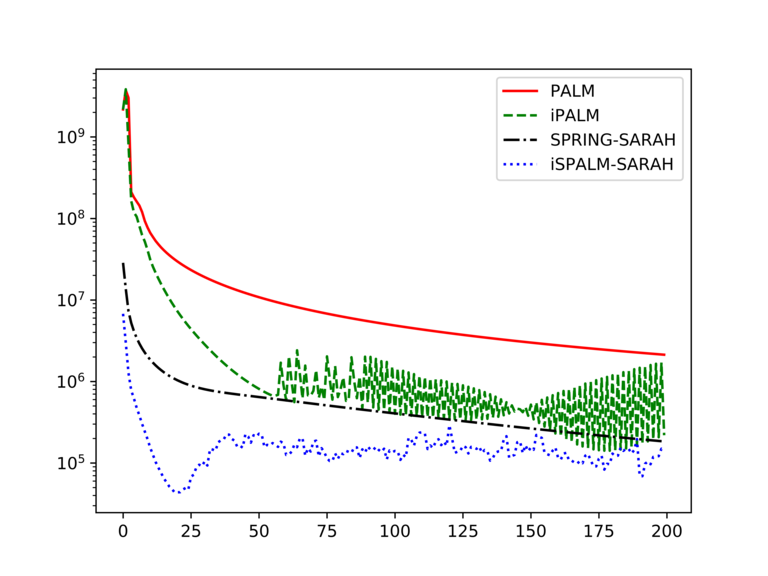}
\caption{Riemannian gradient of the loss versus epochs on the training set.}
\end{subfigure}
\begin{subfigure}[t]{0.5\textwidth}
\centering
\includegraphics[width=\textwidth]{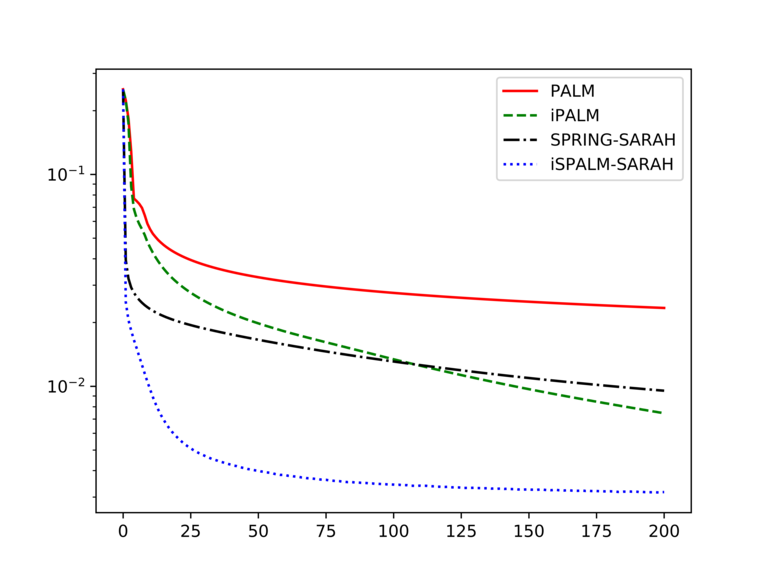}
\caption{Average loss versus epochs on the test set.}
\end{subfigure}\hfill
\begin{subfigure}[t]{0.5\textwidth}
\centering
\includegraphics[width=\textwidth]{imgs/PNNs_train_times}
\caption{Average loss versus execution time on the test set.}
\end{subfigure}
\caption{Loss function versus number of epochs and versus execution time for training a PNN for MNIST classification.}
\label{fig:results_PNNs}
\end{figure}

%---------------------------------------------
\section{Conclusions} \label{sec:concl}
%---------------------------------------------
We combined a stochastic variant of the PALM algorithm 
with the inertial PALM algorithm 
to a new algorithm, called iSPALM. 
We analyzed the convergence behavior of iSPALM and proved convergence results,
if the gradient estimators is  inertial variance-reduced.
In particular, we showed that the expected distance of the subdifferential to zero converges to zero for the sequence of iterates generated by iSPALM. Additionally the sequence of function values achieves linear convergence for functions satisfying a global error bound. 
We proved that a modified version of the negative log-likelihood function of Student-$t$ MMs 
fulfills all necessary convergence assumption 
of PALM, iPALM.
We demonstrated the performance of iSPALM for two quite different applications. 
In the numerical comparison, it turns out that iSPALM shows the best performance of all four algorithms.
In particular, the example with the PNNs demonstrates the importance of combining inertial parameters and stochastic gradient estimators for learning applications.
%although the improvement by using the inertial variant is in the stochastic setting lower than in the deterministic one.

For future work, it would be interesting to compare the performance of the iSPALM algorithm 
with more classical algorithms for estimating the parameters of Student-$t$ MMs, 
in particular with the EM algorithm 
and some of its accelerations. 
For first experiments in this direction we refer to our work \cite{HHLS2019,Hertrich2020}. 

Further, Driggs et al. \cite{DTLDS2020} proved tighter convergence rates for SPRING if the objective function is semi-algebraic. 
Whether these convergence rates also hold true for iSPALM is still open. 

Finally, we intend to apply iSPALM to other practical problems as e.g.\ in more sophisticated examples of deep learning.

%%%%%%%%%%%%%%%%%%%%%%%%%%%%%%%%%%%%%%%%%%%%%%%%%%%%%%%%%%%%%%%%%%%%%%%%%%%%%%%%%%%%%%%%%%%%%%%%%%%%%%%%%%%%%%%%%%%%%%%%%%
\appendix
%-------------------------------------------------
\section{KL Functions}\label{sec:KL}
%-------------------------------------------------
Finally, let us recall the notation of Kurdyka-{\L}ojasiewicz functions.
For $\eta\in(0,\infty]$, we denote by $\Phi_\eta$ the set of all concave continuous functions $\phi\colon[0,\eta)\to\R_{\geq 0}$ which fulfill the following properties:
\begin{enumerate}[(i)]
\item $\phi(0)=0$.
\item $\phi$ is continuously differentiable on $(0,\eta)$.
\item For all $s\in(0,\eta)$ it holds $\phi'(s)>0$.
\end{enumerate}

\begin{Definition}[Kurdyka-{\L}ojasiewicz property]
A  proper, lower semicontinuous function $\sigma\colon\R^d\to(-\infty,+\infty]$  
has the Kurdyka-{\L}ojasieweicz (KL) property at $\bar u\in\dom\partial\sigma=\{u\in\R^d:\partial\sigma\neq\emptyset\}$ 
if there exist $\eta\in(0,\infty]$, a neighborhood $U$ of $\bar u$ and a function $\phi\in\Phi_\eta$, such that for all
$$
u\in U\cap\{v\in\R^d:\sigma(\bar u) < \sigma(v) < \sigma(\bar u)+\eta\},
$$
it holds
$$
\phi'(\sigma(u)-\sigma(\bar u))\dist(0,\partial\sigma(u))\geq 1.
$$
We say that $\sigma$ is a KL function, if it satisfies the KL property in each point $u\in\dom\partial\sigma$.
\end{Definition}

%--------------------------------------------------
\section{Proof of Proposition \ref{prop:sarah-ivr}} \label{app:sarah}
%---------------------------------------------------------------
The proof follows the path of those in \cite[Proposition 2.2]{DTLDS2020}. 
Let $\E_{k,p}=\E(\cdot|(x_1^1,x_2^1),...,(x_1^k,x_2^k),p_1^k)$ denote the expectation conditioned on the first $k$ iterations 
and the event that we do not compute the full gradient at the $k$-th iteration in \eqref{sarah}, $k \ge 1$.
Then we get
\small{
\begin{align}
\E_{k,p}\left(\tilde\nabla_{x_1}H(z_1^k,x_2^k) \right)
&=\tfrac1b \E_{k,p}
\Big(\sum_{i\in B_1^k}\nabla_{x_1}h_i(z_1^k,x_2^k)-\nabla_{x_1}h_i(z_1^{k-1},x_2^{k-1})\Big)
+\tilde \nabla_{x_1}H(z_1^{k-1},x_2^{k-1})\\
&=\nabla_{x_1} H(z_1^k,x_2^k)-\nabla_{x_1}H(z_1^{k-1},x_2^{k-1})+\tilde\nabla_{x_1} H(z_1^{k-1},x_2^{k-1}), \label{ivr_1}
\end{align}
}
and further
\begin{align}
&\E_{k,p}\left(\|\tilde\nabla_{x_1}H(z_1^k,x_2^k)-\nabla_{x_1}H(z_1^k,x_2^k)\|^2\right)\\
&= 
\E_{k,p}\left(\|\tilde\nabla_{x_1}H(z_1^{k-1},x_2^{k-1})-\nabla_{x_1}H(z_1^{k-1},x_2^{k-1})
+\nabla_{x_1}H(z_1^{k-1},x_2^{k-1})-\nabla_{x_1}H(z_1^k,x_2^k)\right.\\
&\left.+\tilde\nabla_{x_1}H(z_1^k,x_2^k)-\tilde\nabla_{x_1}H(z_1^{k-1},x_2^{k-1})\|^2\right)
\\
&=
\|\tilde\nabla_{x_1}H(z_1^{k-1},x_2^{k-1})-\nabla_{x_1}H(z_1^{k-1},x_2^{k-1})\|^2
+\|\nabla_{x_1}H(z_1^{k-1},x_2^{k-1})-\nabla_{x_1}H(z_1^k,x_2^k)\|^2\\
&+\E_{k,p}\left(\|\tilde\nabla_{x_1}H(z_1^k,x_2^k)-\tilde\nabla_{x_1}H(z_1^{k-1},x_2^{k-1})\|^2\right)\\
&+2\left\langle\tilde\nabla_{x_1}H(z_1^{k-1},x_2^{k-1})-\nabla_{x_1}H(z_1^{k-1},x_2^{k-1}), 
\nabla_{x_1}H(z_1^{k-1},x_2^{k-1})-\nabla_{x_1}H(z_1^k,x_2^k)\right\rangle\\
&+2\left\langle\tilde\nabla_{x_1}H(z_1^{k-1},x_2^{k-1})-\nabla_{x_1}H(z_1^{k-1},x_2^{k-1}),
\E_{k,p}\left(\tilde\nabla_{x_1}H(z_1^k,x_2^k)-\tilde\nabla_{x_1}H(z_1^{k-1},x_2^{k-1})\right)\right\rangle\\
&+2\left\langle\nabla_{x_1}H(z_1^{k-1},x_2^{k-1})-\nabla_{x_1}H(z_1^k,x_2^k),\E_{k,p}
\left(\tilde\nabla_{x_1}H(z_1^k,x_2^k)-\tilde\nabla_{x_1}H(z_1^{k-1},x_2^{k-1})\right)\right\rangle. \label{ivr_2}
\end{align}
By \eqref{ivr_1}, we see that
$$
\E_{k,p}\left(\tilde\nabla_{x_1}H(z_1^k,x_2^k)-\tilde\nabla_{x_1}H(z_1^{k-1},x_2^{k-1})\right)=\nabla_{x_1}H(z_1^k,x_2^k)-\nabla_{x_1} H(z_1^{k-1},x_2^{k-1}).
$$
Thus, the first two inner products in \eqref{ivr_2} sum to zero and the third one is equal to
\begin{align}
&2\left\langle\nabla_{x_1}H(z_1^{k-1},x_2^{k-1})-\nabla_{x_1}H(z_1^k,x_2^k),\E_{k,p}\left(\tilde\nabla_{x_1}H(z_1^k,x_2^k)-\tilde\nabla_{x_1}H(z_1^{k-1},x_2^{k-1})\right)\right\rangle\\
&= 2\left\langle\nabla_{x_1}H(z_1^{k-1},x_2^{k-1})-\nabla_{x_1}H(z_1^k,x_2^k),\nabla_{x_1}H(z_1^k,x_2^k)-\nabla_{x_1}H(z_1^{k-1},x_2^{k-1})\right\rangle\\
&=-2\|\nabla_{x_1} H(z_1^k,x_2^k)-\nabla_{x_1}H(z_1^{k-1},x_2^{k-1})\|^2.
\end{align}
This yields 
\begin{align}
&\E_{k,p}\left(\|\tilde\nabla_{x_1}H(z_1^k,x_2^k)-\nabla_{x_1}H(z_1^k,x_2^k)\|^2\right)\\
&\leq\|\tilde\nabla_{x_1}H(z_1^{k-1},x_2^{k-1})-\nabla_{x_1}H(z_1^{k-1},x_2^{k-1})\|^2-\|\nabla_{x_1}H(z_1^{k-1},x_2^{k-1})-\nabla_{x_1}H(z_1^k,x_2^k)\|^2\\
&+\E_{k,p}\left(\|\tilde\nabla_{x_1}H(z_1^k,x_2^k)-\tilde\nabla_{x_1}H(z_1^{k-1},x_2^{k-1})\|^2\right)\\
&\leq
\|\tilde\nabla_{x_1}H(z_1^{k-1},x_2^{k-1})\!-\!\nabla_{x_1}H(z_1^{k-1},x_2^{k-1})\|^2
+
\E_{k,p}\left(\|\tilde\nabla_{x_1}H(z_1^k,x_2^k)\!-\!\tilde\nabla_{x_1}H(z_1^{k-1},x_2^{k-1})\|^2\right).
\end{align}
Since the function $x\mapsto\|x\|^2$ is convex, the second summand fulfills
{\small
\begin{align}
&\E_{k,p}\left(\|\tilde\nabla_{x_1}H(z_1^k,x_2^k)-\tilde\nabla_{x_1}H(z_1^{k-1},x_2^{k-1})\|^2\right)\\
&=\E_{k,p}\Big(\Big\|\tfrac1b \Big(\sum_{j\in B_1^k}\nabla_{x_1}h_j(z_1^k,x_2^k)-\nabla_{x_1}h_j(z_1^{k-1},x_2^{k-1})\Big)\Big\|^2\Big)\\
&\leq\tfrac1n\sum_{j=1}^n \|\nabla h_j(z_1^k,x_2^k)-\nabla_{x_1}h_j(z_1^{k-1},x_2^{k-1})\|^2,
\end{align}
}
so that we obtain
\begin{align}
\E_{k,p}\left(\|\tilde\nabla_{x_1}H(z_1^k,x_2^k)-\nabla_{x_1}H(z_1^k,x_2^k)\|^2\right)
&\leq
\|\tilde\nabla_{x_1}H(z_1^{k-1},x_2^{k-1})-\nabla_{x_1}H(z_1^{k-1},x_2^{k-1})\|^2\\
&+ \tfrac1n\sum_{j=1}^n\|\nabla_{x_1}h_j(z_1^k,x_2^k)-\nabla_{x_1}h_j(z_1^{k-1},x_2^{k-1})\|^2.
\end{align}
Since the conditional expectation $\E_k$  
of $\|\tilde\nabla_{x_1}H(z_1^k,x_2^k)-\nabla_{x_1}H(z_1^k,x_2^k)\|^2$ conditioned on the event 
that the full gradient is computed in \eqref{sarah} is zero, 
and taking the $M$-Lipschitz continuity of the gradients of the $h_j$ into account,
we get
\begin{align}
&\E_{k}\left(\|\tilde\nabla_{x_1}H(z_1^k,x_2^k)-\nabla_{x_1}H(z_1^k,x_2^k)\|^2\right)\\
&\leq(1-\tfrac1p)\Big(\|\tilde\nabla_{x_1}H(z_1^{k-1},x_2^{k-1})-\nabla_{x_1}H(z_1^{k-1},x_2^{k-1})\|^2\\
&+\tfrac1n\sum_{j=1}^n \|\nabla_{x_1}h_j(z_1^k,x_2^k)-\nabla_{x_1}h_j(z_1^{k-1},x_2^{k-1})\|^2\Big)\\
&\leq  
(1-\tfrac1p) \left(\|\tilde\nabla_{x_1}H(z_1^{k-1},x_2^{k-1})-\nabla_{x_1}H(z_1^{k-1},x_2^{k-1})\|^2+M^2\|(z_1^k,x_2^k)-(z_1^{k-1},x_2^{k-1})\|^2\right).
\end{align}
By symmetric arguments, it holds 
\begin{align}
&\E_{k}\left(\|\tilde\nabla_{x_2}H(x_1^{k+1},z_2^k)-\nabla_{x_2}H(x_1^{k+1},z_2^k)\|^2\right)\\
\leq& (1-\tfrac1p) \left(\|\tilde\nabla_{x_2}H(x_1^{k},z_2^{k-1})-\nabla_{x_2}H(x_1^{k},z_2^{k-1})\|^2+M^2\E_k(\|(x_1^{k+1},z_2^k)
-(x_1^{k},z_2^{k-1})\|^2) \right).
\end{align}
Using $(a+b+c)^2 \le 3(a^2+b^2+c^2)$ and Lemma \ref{prop:deltas}, we can estimate
\begin{align}
&\|(z_1^k,x_2^k)-(z_1^{k-1},x_2^{k-1})\|^2 \,  = \, \|z_1^k-z_1^{k-1}\|^2+\|x_2^k-x_2^{k-1}\|^2\\
&\leq
3\|z_1^k-x_1^k\|^2+3\|x_1^k-x_1^{k-1}\|^2+3\|x_1^{k-1}-z_1^{k-1}\|^2+\|x_2^k-x_2^{k-1}\|^2\\
&\leq
3(1+( \beta_1^k)^2)\|x_1^k-x_1^{k-1}\|^2+3(\beta_1^{k-1})^2\|x_1^{k-1}-x_1^{k-2}\|^2+\|x_2^k-x_2^{k-1}\|^2.
\end{align}
Further, we have 
\begin{align}
&\E_k(\|(x_1^{k+1},z_2^k)-(x_1^{k},z_2^{k-1})\|^2) \,  = \,
\E_k(\|x_1^{k+1}-x_1^k\|^2)+\|z_2^k-z_2^{k-1}\|^2\\
&\leq \E_k(\|x_1^{k+1}-x_1^k\|^2)+3\|z_2^k-x_2^k\|^2+3\|x_2^k-x_2^{k-1}\|^2+3\|x_2^{k-1}-z_2^{k-1}\|^2\\
&\leq \E_k(\|x_1^{k+1}-x_1^k\|^2)+3(1+( \beta_2^k)^2)\|x_2^k-x_2^{k-1}\|^2+3(\beta_2^{k-1})^2\|x_2^{k-1}-x_2^{k-2}\|^2.
\end{align}
Altogether we obtain for
$$
\Upsilon_{k+1}\coloneqq \|\tilde\nabla_{x_1}H(z_1^k,x_2^k)-\nabla_{x_1}H(z_1^k,x_2^k)\|^2
+\|\tilde\nabla_{x_2}H(x_1^{k+1},z_2^k)-\nabla_{x_2}H(x_1^{k+1},z_2^k)\|^2.
$$
that
\begin{align}
&\E_k(\Upsilon_{k+1}) =\E_{k}\left(\|\tilde\nabla_{x_1}H(z_1^k,x_2^k)
-\nabla_{x_1}H(z_1^k,x_2^k)\|^2+\|\tilde\nabla_{x_2}H(x_1^{k+1},z_2^k)-\nabla_{x_2}H(x_1^{k+1},z_2^k)\|^2\right)\\
&\leq 
(1-\tfrac1p)\Upsilon_k+ V_\Upsilon (\E_k(\|x^{k+1}-x^{k}\|^2)+\|x^k-x^{k-1}\|^2+\|x^{k-1}-x^{k-2}\|^2),\label{eq:in_var_red_ineq}
\end{align}
where $V_\Upsilon = 3(1-\tfrac1p)M^2 \left( 1+\max( (\bar\beta_1)^2,(\bar \beta_2)^2 ) \right)$.
This proves the properties (i) and (ii) of Definition \ref{def:ivr}.
Taking the full expectation in  \eqref{eq:in_var_red_ineq} and iterating, we  
\begin{align}
\E(\Upsilon_k) 
&\leq (1-\tfrac1p)^{k-1} \E(\Upsilon_1) \\
&+
V_\Upsilon 
\sum_{l=1}^{k-1} (1-\tfrac1p)^{k-l-1}\E\left(\|x^{l+1}-x^{l}\|^2+\|x^{l}-x^{l-1}\|^2+\|x^{l-1}-x^{l-2}\|^2\right).
\end{align}
We want to show that $\E(\Upsilon_k)\to 0$ as $k\to\infty$, if $\E(\|x^{k}-x^{k-1}\|^2)\to0$ as $k\to\infty$.
Since the first summand converges to zero for $k$ large enough, 
it remains to prove that for an arbitrary $\epsilon>0$, 
there exists some $k_0\in\N$ such that for all $k\geq k_0$ the sum becomes not larger than $\epsilon$
Recall that $\sum_{l=0}^\infty (1-\tfrac1p)^l=p$. 
Now we choose $k_1\in\N$ such that for all $k\geq k_1$ we have 
$\E(\|x^{k+1}-x^{k}\|^2)<\tfrac\epsilon{2 p}$. 
Further we define $k_2\in\N$ such that $(1-\tfrac1p)^{k_2}<\tfrac{\epsilon}{6S p}$, 
where $S \coloneqq \max_{k\in\N}\E(\|x^{k}-x^{k-1}\|^2)$. 
Then the above sum can be estimated for $k\geq k_0\coloneqq k_1+k_2$ as
\begin{align}
&\sum_{l=1}^{k-1} (1-\tfrac1p)^{k-l-1}\E\left(\|x^{l+1}-x^l\|^2+\|x^l-x^{l-1}\|^2+\|x^{l-1}-x^{l-2}\|^2\right)\\
&= \sum_{l=1}^{k_1}(1-\tfrac1p)^{k-l-1}\E\left(\|x^{l+1}-x^l\|^2+\|x^l-x^{l-1}\|^2+\|x^{l-1}-x^{l-2}\|^2\right)\\
&+\sum_{l=k_1+1}^{k-1}(1-\tfrac1p)^{k-l-1}\E\left(\|x^{l+1}-x^l\|^2+\|x^l-x^{l-1}\|^2+\|x^{l-1}-x^{l-2}\|^2\right)\\
&\leq (1-\tfrac1p)^{k-k_1}\sum_{l=1}^{k_1}(1-\tfrac1p)^{k_1-l}\E\left(\|x^{l+1}-x^l\|^2+\|x^l-x^{l-1}\|^2+\|x^{l-1}-x^{l-2}\|^2\right)\\
&+\sum_{l=k_1+1}^{k}(1-\tfrac1p)^{k-l-1}\tfrac\epsilon{2p} 
\leq \epsilon,
\end{align}
and we are done.

%--------------------------------------------------
\section{Derivatives of the Likelihood of Student-$t$ MMs}\label{sec:derivatives_iPALM}
%-------------------------------------------------
In this section, we compute the outer derivatives of the objective function in \eqref{eq:unconstrainted_problem}
which we need in the numerical computations and in the proof of Lemma \ref{lem:Lipschitz_Likelihood}. 
In \cite{HHLS2019}, the derivatives of $g(\nu,\mu,\Sigma) \coloneqq \log(f(x|\nu,\mu,\Sigma)$ 
were computed as follows:
\begin{align}
\frac{\partial g}{\partial \mu}(\nu,\mu,\Sigma) 
& = \frac{d+\nu}{\nu + s}  \,\Sigma^{-1}(x-\mu), \label{der_1}
\\
\frac{\partial g}{\partial \Sigma}(\nu,\mu,\Sigma)	
& = \frac12 \left( \frac{d+\nu}{\nu + s}  \, \Sigma^{-1}(x-\mu)(x-\mu)^\tT \Sigma^{-1} - \Sigma^{-1} \right),\label{der_2}
\\
\frac{\partial g}{\partial \nu}(\nu,\mu,\Sigma ) 
& = \frac12
\left(
\psi \left(\frac{\nu + d}{2}\right) - \psi \left(\frac{\nu}{2}\right) 
-  
\frac{d-s }{\nu +  s}
- 
\log\left( 1 + \frac{s}{\nu} \right),
\right)\label{der_3}
\end{align}
where $s \coloneqq (x-\mu)^\tT \Sigma^{-1} (x-\mu)$ and $\Psi$ is the
the \emph{digamma function} defined by
$$
\psi(x) \coloneqq \frac{\mathrm{d}}{\mathrm{d}x}\log\left(\Gamma(x)\right) = \frac{\Gamma'(x)}{\Gamma(x)}.
$$
We use the abbreviations
$$
f_{i,k} \coloneqq f(x_i|\nu_k,\mu_k,\Sigma_k),\quad
\gamma_i \coloneqq \Big( \sum_{k=1}^K \alpha_k f_{i,k} \Big)^{-1}, \quad
s_{i,k} \coloneqq (x_i-\mu_k)^T\Sigma_k^{-1}(x_i-\mu_k). 
$$
Then we obtain for 
$$
\mathcal L(\alpha,\nu,\mu,\Sigma|\mathcal X)=-\sum_{i=1}^n\log \bigg(\sum_{k=1}^K\alpha_k f_{i,k}\bigg)
$$
that the derivative with respect to $\alpha$ is given by
\begin{align} \label{eq:der_alpha}
\frac{\partial \mathcal L(\alpha,\nu,\mu,\Sigma|\mathcal X)}{\partial \alpha_l}
=-\sum_{i=1}^n\gamma_i f_{i,k}
\end{align}
Using that for $g= \log f$, the relation $g' f = f'$ holds true,
the derivatives with respect to $\mu_l,\Sigma_l$ and $\nu_l$ have the form
\begin{align}
\frac{\partial \mathcal L(\alpha,\nu,\mu,\Sigma|\mathcal X)}{\partial \bullet_l}
= -\sum_{i=1}^n \gamma_i \alpha_l f_{i,l} \frac{\partial f(\alpha,\nu,\mu,\Sigma|\mathcal X)}{\partial \bullet_l}.
\end{align}
Together with \eqref{der_1} - \eqref{der_3}, we obtain
\begin{align}
\nabla_{\mu_l} \mathcal L(\alpha,\nu,\mu,\Sigma|\mathcal X)
&=\sum_{i=1}^n\gamma_i \alpha_l f_{i,l} \frac{d+\nu_l}{\nu_l+s_{i,l}} \, \Sigma_l^{-1}(\mu_l-x_i),\label{eq:der_mu}
\\
\nabla_{\Sigma_l}\mathcal L(\alpha,\nu,\mu,\Sigma|\mathcal X)
&=
\frac12\sum_{i=1}^n\gamma_i\alpha_l f_{i,l}
\bigg(\Sigma_l^{-1}-\frac{d+\nu_l}{\nu_l+s_{i,l}}\, \Sigma^{-1}(x_i-\mu_l)(x_i-\mu_l)^T\Sigma^{-1}\bigg),\label{eq:der_sigma}
\\
\frac{\partial \mathcal L(\alpha,\nu,\mu,\Sigma|\mathcal X)}{\partial \nu_l}
&=
\frac12\sum_{i=1}^n\gamma _i\alpha_l f_{i,l}
\left( \psi \left(\frac{\nu_l}{2}\right)
-\psi \left(\frac{\nu_l + d}{2}\right) +
\frac{d-s_{i,l} }{\nu_l +  s_{i,l}}
+ 
\log\left( 1+ \frac{s_{i,l}}{\nu_l} \right)
\right).\label{eq:der_nu}
\end{align}

%--------------------------------------------------
\section{Proof of Lemma \ref{lem:Lipschitz_Likelihood} }\label{sec:proof}
%-------------------------------------------------
Since the sum of Lipschitz continuous functions is Lipschitz continuous, 
it is sufficient to show the claim for the summands of $H$ in \eqref{eq:unconstrainted_problem}.
Hence we consider only
$$h(\alpha,\nu,\mu,\Sigma)
\coloneqq
\log \Big(\sum_{k=1}^K \tilde \alpha_k f_k \Big), 
\qquad f_k \coloneqq f(x|\tilde\nu_k,\tilde\mu_k,\tilde\Sigma_k),
$$
where 
$\tilde \alpha=\varphi_1(\alpha)$, 
$\tilde\nu = \varphi_2 (\nu)$, 
$\tilde\Sigma=\varphi_3 (\Sigma)$ are given by \eqref{trafo}.
Set
$$
\gamma \coloneqq \left(\sum_{k=1}^K \tilde \alpha_k f_{k}\right)^{-1},\quad
s_{k} \coloneqq (x-\tilde \mu_k)^T \tilde \Sigma_k^{-1}(x-\tilde \mu_k). 
$$

%----------------

1. By \eqref{eq:der_alpha} we obtain
\begin{align}
\frac{\partial h(\alpha,\nu,\mu,\Sigma)}{\partial\alpha_l}
=
\frac{\exp(\alpha_l)f_{l}}{\sum_{k=1}^K \exp(\alpha_k)f_{k}} - \frac{\exp(\alpha_l)}{\sum_{k=1}^K \exp(\alpha_k)}
\end{align}
and further for the Hessian of $h$ with respect to $\alpha$,
\begin{align}
\frac{\partial h(\alpha,\nu,\mu,\Sigma)}{\partial \alpha_l \partial\alpha_j}
&=
\delta_{j,l} \left( 
\frac{\exp(\alpha_l) f_{l}}{\sum_{k=1}^K\exp(\alpha_k)f_{k}} 
-
\frac{\exp(\alpha_l)}{\sum_{k=1}^K \exp(\alpha_k)}
\right)\\
&- 
\frac{\exp(\alpha_l) \exp(\alpha_j) f_{l} f_{j} }{\Big(\sum_{k=1}^K \exp(\alpha_k)f_{k} \Big)^2} 
+ \frac{\exp(\alpha_l) \exp(\alpha_j)}{\Big(\sum_{k=1}^K \exp(\alpha_k) \Big)^2}  .
\end{align}
This is bounded so that $\nabla_\alpha h(\cdot,\nu,\mu,\Sigma)$ is globally Lipschitz continuous.
%---------------

2. Using \eqref{der_3}, we get
\begin{align}
\frac{\partial}{\partial \nu_l} h(\alpha,\nu,\mu,\Sigma)
= \tilde \alpha_l \underbrace{\gamma f_{l}}_{g_1}
\underbrace{\Big(
\psi\left( \psi\left(\frac{d+\tilde \nu_l}{2}\right) - \frac{\tilde \nu_l}{2}\right)
-
\frac{d-s_l}{\tilde \nu_l +  s_l}
-
\log\left( 1 + \frac{s_l}{\tilde \nu_l} \right)
\Big)
\, \nu_l}_{g_2}.
\end{align}
We show that the functions $g_i$, $i=1,2$ are Lipschitz continuous and bounded.
This implies that $\frac{\partial}{\partial \nu_l} h(\alpha,\cdot,\mu,\Sigma)$ is Lipschitz continuous.
It holds 
\begin{align}\label{gut}
|g_2(\nu_l)|
\leq
|\nu_l|\, |\psi\left(\frac{d+\tilde \nu_l}{2} \right)-\psi\left(\frac{\tilde\nu_l}{2}\right)|
+
|\nu_l| \left|\frac{d-s_l}{\tilde \nu_l +  s_l} \right|
+
\log  \left(1+ \frac{s_l}{\tilde \nu_l} \right)^{|\nu_l|}.
\end{align}
Using the summation formula 
\begin{equation} \label{psi_1}
\psi(x+1)=\psi(x)+\tfrac1x
\end{equation} 
and the fact that the digamma function is monotone increasing we conclude
\begin{align}\label{eq:abs_g_nu}
|g_2(\nu)|
&\le
\sum_{r= 1}^{\lceil\tfrac{d}2\rceil} 
\frac{|\nu_l|}{\tfrac{d+\tilde\nu_l}{2} -r}
+
|\nu_l|\Big|\frac{d-s_l}{\tilde \nu_l +  s_l}\Big|
+
\log \left(1+ \frac{s_l}{\tilde \nu_l} \right)^{|\nu_l|}\\
&\le
\sum_{r=1}^{\lceil\tfrac{d}2\rceil} 
\frac{|\nu_l|}{\tfrac{d+ \nu_l^2 + \epsilon}{2}-r}
+
|\nu_l|\Big|\frac{d-s_l}{\nu_l^2 + \epsilon +  s_l}\Big|
+
\log \left(1+ \frac{s_l}{\nu_l^2} \right)^{|\nu_l|}.
\end{align}
Since
\begin{align}
\lim_{\nu_l\to\pm\infty}
\log (1+\tfrac1{\nu^2}s_l)^{|\nu|}
\leq
\lim_{\nu_l\to\pm\infty}\log \left(1+\tfrac1{\nu^2}s_l \right)^{\nu^2}
=s_l
\end{align}
and $g_2$ is continuous we conclude that $g_2$ is bounded. 
Further, it holds 
\begin{align}
g_2'(\nu_l)
&=
\psi\left( \frac{d+\tilde \nu_l}{2}\right) - \psi\left(\frac{\tilde \nu_l}{2}\right)
- 
\frac{d-s_l}{\tilde \nu_l +  s_l} - \log\left(1+\ \frac{s_l}{\tilde \nu_l} \right)
\\
&+
\nu_l^2 \left( \psi'\left( \frac{d+\tilde \nu_l}{2} \right) - \psi'\left(\frac{\tilde \nu_l}{2} \right)  \right) 
+
\frac{2(d-s_l)\nu_l}{(\tilde \nu_l + s_l)^2} + \frac{2s_l \nu_l}{(\tilde \nu_l^2 + s_l)^2 + s_l (\tilde \nu_l^2 + s_l)}.
\end{align}
Using again \eqref{psi_1}  and 
$\psi'(x+1)=\psi'(x)+\tfrac1{x^2}$ as well as 
the fact that the digamma function and its derivatives are monotone increasing we get, 
\begin{align}
\lim_{\nu_l\to\pm\infty}
|g_2'(\nu_l)|
&\le
\lim_{\nu_l\to\pm\infty} \Big( \sum_{r=0}^{\lceil\tfrac{d}2\rceil} 
\frac{1}{\tfrac{\tilde \nu_l}{2}+r}
+
\sum_{r=0}^{\lceil\tfrac{d}2\rceil} 
\frac{|\nu_l|^2}{(\tfrac{\tilde \nu_l}{2}+r)^2} \Big) = 0.
\end{align}
Since $g_2'(\nu)$ is continuous, this yields that it is bounded which implies that 
$g_2$ is Lipschitz continuous.

The function $g_1$ is obviously bounded.
Further we obtain for $j \not = l$ that
\begin{align}
\nabla_{\nu_j}g_1(\nu)
= 
-2 \tilde \alpha_j f_{l}\gamma^2 \frac{\partial f_j}{\partial \nu_j} 
=
-2 g_2 (\nu_j) \tilde \alpha_j f_{l} f_{j} \gamma^2 
\end{align}
so that
$$
|\nabla_{\nu_j}g_1(\nu_l)| = 2 | g_2 (\nu_j)|\,  |\tilde \alpha_j f_{l} f_{j} \gamma^2|.
$$
This expression is bounded. Similarly, we get for $j=l$ that
\begin{align}
|\nabla_{\nu_l}g_1(\nu_l)| = 
2 | g_2 (\nu_l)| \,  |\tilde \alpha_l f_{l}^2 \gamma^2 + \gamma f_l |.
\end{align}
Thus, $g_1$ is Lipschitz continuous. 
%---------------

3. 
By  \eqref{eq:der_mu} we obtain 
\begin{align}
\nabla_{\mu_l}h(\alpha,\nu,\mu,\Sigma)
= \tilde \alpha_l \underbrace{\gamma f_{l}}_{g_1} \underbrace{\frac{d+\tilde \nu_l}{\tilde \nu_l+s_l}\, \tilde \Sigma_l^{-1}(x-\mu_l)}_{g_2}.
\end{align}
As in the second part of the proof, it suffices to show that $g_i$, $i=1,2$
are bounded and Lipschitz continuous. 
Calculating the Jacobian
\begin{align}
\nabla_{\mu_l} g_2(\mu_l)
=
\frac{d+\tilde \nu_l}{(\tilde \nu_l + s_l)^2}
\tilde \Sigma_l^{-1}(\mu_l-x)(\mu_l-x)^T \tilde \Sigma_l^{-1}
+\frac{d+\tilde \nu_l}{\tilde \nu_l +s_l}\tilde \Sigma_l^{-1}
\end{align}
and taking the Frobenius norm $\| \cdot\|_F$, we obtain
\begin{align}
\| \nabla_{\mu_l} g_2(\mu_l)\|_F 
&\leq
\frac{d+\tilde \nu_l}{(\tilde \nu_l+s_l)^2}
\|\tilde \Sigma_l^{-1} (\mu_l-x)\|^2  +\mathrm{const}\\
&\leq
\frac{d+\tilde \nu_l}{\left(\frac{\tilde \nu_l}{\|\tilde\Sigma_l^{-\frac12}(\mu-x)\|}
+
\|\tilde\Sigma_l^{-\frac12}(\mu-x)\|\right)^2} \|\tilde\Sigma_l^{-1}\|_F +\mathrm{const}\\
&\leq
\frac{d+\tilde \nu_l}{(2\min(1,\tilde \nu_l))^2}\|\tilde\Sigma_l^{-1}\|_F +\mathrm{const}.
\end{align}
Thus $g_2$ is Lipschitz continuous. 
Since
\begin{align}
\|g_2(\mu_l)\|
&\leq
\frac{d+\tilde \nu_l}{\tilde \nu_l+s_l}\|\tilde \Sigma_l^{-\frac12}\|_F \|\tilde \Sigma_l^{-\frac12}(\mu_l-x)\|\\
&=
\frac{d+\tilde \nu_l}{\frac{\tilde \nu_l}{\|\tilde \Sigma_l^{-\frac12}(\mu_l-x)\|}
+
\|\tilde \Sigma_l^{-\frac12}(\mu_l-x)\|}\|\tilde \Sigma_l^{-\frac12}\|
\leq
\frac{d+\tilde \nu_l}{2\min(1,\tilde \nu_l)}\|\tilde \Sigma_l^{-\frac12}\|,
\end{align}
 $g_2$ is bounded.

The function $g_1$ is obviously bounded.
Further, it holds
\begin{align}
\|\nabla_{\mu_j}g_1(\mu)\| &=
\left\{
\begin{array}{ll}
\|g_2(\mu_l)\| \, |\gamma^2 f_{l} f_j \tilde \alpha_j | &\mathrm{if} \; j\neq l\\
 \|g_2(\mu_l)\| \, |\gamma^2 f_{l}^2  \tilde \alpha_l + \gamma f_l|  &\mathrm{if} \; j= l,
\end{array}
\right.
\end{align}
so that $g_1$ is Lipschitz continuous.                                                  
%--------------

4. At the end, we use \eqref{der_2} to compute
\begin{equation}
\nabla_{\Sigma_l}h (\alpha,\nu,\mu,\Sigma)
= 
\tilde \alpha_l \underbrace{\gamma f_{l}  }_{g_1}
\underbrace{ \left(\frac{d+\tilde \nu_l}{\tilde \nu_l+s_l} 
\tilde \Sigma_l^{-1}(x-\mu_l)(x-\mu_l)^T \tilde \Sigma_l^{-1}  - \tilde \Sigma_l^{-1} \right) \Sigma_l}_{g_2}.
\end{equation}
We show that $g_i$, $i=1,2$ are bounded and Lipschitz continuous. 
We have 
\begin{align}
g_2(\Sigma_l)
&=\underbrace{\bigg(\frac{d+\tilde \nu_l}{\tilde \nu_l+s_l} \tilde \Sigma_l^{-1}(x-\mu_l)(x-\mu_l)^T - I_d\bigg)}_{h_1(\Sigma_l)}
\underbrace{\tilde \Sigma_l^{-1}\Sigma_l}_{h_2(\Sigma_l)}
\end{align}
Obviously, $h_1$ is bounded. The second factor $h_2$ is bounded,
since with the spectral decomposition $\Sigma_l=PDP^T$ it holds
$$
\tilde \Sigma_l^{-1}\Sigma_l = (P(D^2+\epsilon I)P^T)^{-1}PDP^T=P(D^2+\epsilon I)^{-1}DP^T,
$$
so that the absolute value of the largest eigenvalue of $h_2(\Sigma_l)$ is smaller than $1$. 

To prove the Lipschitz continuity of $g_2$ we compute the directional derivative using the computation rules from \cite{PP08}:
\begin{align}
D_{\Sigma_l}(\tilde\Sigma_l^{-1})[H]
=
D_{\Sigma_l}((\Sigma_l^2+\epsilon I_d)^{-1})[H]
=
\tilde \Sigma_l^{-1} D_{\Sigma_l}(\Sigma_l^2+\epsilon I_d)\tilde \Sigma_l^{-1}
=2 \tilde \Sigma_l^{-1} \Sigma_l \tilde \Sigma_l^{-1}.\label{eq:der_sig_m1}
\end{align}
Then we obtain
\begin{align}
\|D_{\Sigma_l}h_2[H]\|_F
=
\|2\tilde \Sigma_l^{-1} \Sigma_l \tilde\Sigma_l^{-1}\Sigma_l + \tilde\Sigma_l^{-1}\|_F
\leq 
4\|h_2(\Sigma_l)\|_F^2+\|\tilde \Sigma_l^{-1}\|_F
\end{align}
which is bounded. Thus, $h_2$ is Lipschitz continuous.
 
To show, that $h_1$ is Lipschitz continuous, note that the mapping
$x\mapsto \frac{d+\nu}{\nu+x}$ has a bounded derivative and is Lipschitz continuous. 
Further, the mapping $A\mapsto (x-\mu) A (x-\mu)$ has a bounded derivative, if $A$ is bounded. 
Together with the fact that 
$\Sigma_l \mapsto \tilde \Sigma_l^{-1}$ has a bounded derivative by \eqref{eq:der_sig_m1} and is bounded, 
this yields that the mapping
\begin{align}
\Sigma_l \mapsto\frac{d+\tilde \nu_l}{\tilde \nu_l+(x-\mu_l)\tilde \Sigma_l^{-1}(x-\mu_l)} \leq \frac{d+\tilde \nu_l}{\tilde \nu_l}
\end{align}
is Lipschitz continuous as a concatenation of Lipschitz continuous functions. 
Since also $\Sigma_l \mapsto \tilde \Sigma_l^{-1}$ is Lipschitz continuous by \eqref{eq:der_sig_m1} and bounded, 
we get that $h_1$ is Lipschitz continuous. Now, $h_1$ and $h_2$ are Lipschitz continuous and bounded. 
Thus also $g_2 = h_1 h_2$ is Lipschitz continuous and bounded.

Finally, the function $g_1$ maps into the interval $[0,1]$ and is Lipschitz continuous by the same arguments as in the second part of the proof.
$\Box$

%-----------------------------------------------------------
\section*{Acknowledgment}
The authors want to thank T. Pock (TU Graz) for fruitful discussions on iPALM.
\\
 Funding by the German Research Foundation (DFG) with\-in the project STE 571/16-1 is gratefully acknowledged.

%-------------------------------------------------
\bibliographystyle{abbrv}
\bibliography{Student_t}

\begin{thebibliography}{10}

\bibitem{AB2009}
H.~Attouch and J.~Bolte.
\newblock On the convergence of the proximal algorithm for nonsmooth functions
  involving analytic features.
\newblock {\em Mathematical Programming. A Publication of the Mathematical
  Programming Society}, 116(1-2, Ser. B):5--16, 2009.

\bibitem{ABRS2010}
H.~Attouch, J.~Bolte, P.~Redont, and A.~Soubeyran.
\newblock Proximal alternating minimization and projection methods for
  nonconvex problems: An approach based on the {K}urdyka-{\l}ojasiewicz
  inequality.
\newblock {\em Mathematics of Operations Research}, 35(2):438--457, 2010.

\bibitem{BM18}
A.~Banerjee and P.~Maji.
\newblock Spatially constrained {S}tudent's $t$-distribution based mixture
  model for robust image segmentation.
\newblock {\em Journal of Mathematical Imaging Vision}, 60(3):355--381, 2018.

\bibitem{BST2014}
J.~Bolte, S.~Sabach, and M.~Teboulle.
\newblock Proximal alternating linearized minimization for nonconvex and
  nonsmooth problems.
\newblock {\em Mathematical Programming}, 146(1-2, Ser. A):459--494, 2014.

\bibitem{Bottou2010}
L.~Bottou.
\newblock Large-scale machine learning with stochastic gradient descent.
\newblock In {\em Proceedings of COMPSTAT’2010}, volume~1, pages 177--186.
  Springer, 2010.

\bibitem{Byrne2017}
C.~L. Byrne.
\newblock {\em The {EM} Algorithm: Theory, Applications and Related Methods}.
\newblock Lecture Notes, University of Massachusetts, 2017.

\bibitem{CM2009}
O.~Capp{\'e} and E.~Moulines.
\newblock On-line expectation--maximization algorithm for latent data models.
\newblock {\em Journal of the Royal Statistical Society: Series B (Statistical
  Methodology)}, 71(3):593--613, 2009.

\bibitem{CERS2018}
A.~Chambolle, M.-J. Ehrhardt, P.~Richt\'arik, and C.-B. Schoenlieb.
\newblock Stochastic primal-dual hybrid gradient algorithm with arbitrary
  sampling and imaging applications.
\newblock {\em SIAM Journal on Optimization}, 2018.

\bibitem{CP11}
A.~Chambolle and T.~Pock.
\newblock A first-order primal-dual algorithm for convex problems with
  applications to imaging.
\newblock {\em Journal of Mathematical Imaging and Vision}, 40(1):120--145,
  2011.

\bibitem{TZTZ2018}
J.~Chen, J.~Zhu, Y.~W. Teh, and T.~Zhang.
\newblock Stochastic expectation maximization with variance reduction.
\newblock In S.~Bengio, H.~Wallach, H.~Larochelle, K.~Grauman, N.~Cesa-Bianchi,
  and R.~Garnett, editors, {\em Advances in Neural Information Processing
  Systems 31}, pages 7967--7977. Curran Associates, Inc., 2018.

\bibitem{DEU2016}
D.~Davis, B.~Edmunds, and M.~Udell.
\newblock The sound of apalm clapping: Faster nonsmooth nonconvex optimization
  with stochastic asynchronous palm.
\newblock In {\em Advances in Neural Information Processing Systems}, pages
  226--234, 2016.

\bibitem{DBL2014}
A.~Defazio, F.~Bach, and S.~Lacoste-Julien.
\newblock Saga: A fast incremental gradient method with support for
  non-strongly convex composite objectives.
\newblock In {\em Advances in Neural Information Processing Systems}, pages
  1646--1654, 2014.

\bibitem{DHWMZ2019}
M.~Ding, T.~Huang, S.~Wang, J.~Mei, and X.~Zhao.
\newblock Total variation with overlapping group sparsity for deblurring images
  under {C}auchy noise.
\newblock {\em Applied Mathematics and Computation}, 341:128--147, 2019.

\bibitem{DTLDS2020}
D.~Driggs, J.~Tang, J.~Liang, M.~Davies, and C.-B. Schönlieb.
\newblock {SPRING}: A fast stochastic proximal alternating method for
  non-smooth non-convex optimization.
\newblock {\em ArXiv preprint arXiv:2002.12266}, 2020.

\bibitem{GNL09}
D.~Gerogiannis, C.~Nikou, and A.~Likas.
\newblock The mixtures of {S}tudent’s $t$-distributions as a robust framework
  for rigid registration.
\newblock {\em Image and Vision Computing}, 27(9):1285--1294, 2009.

\bibitem{GLZX2019}
I.~Gitman, H.~Lang, P.~Zhang, and L.~Xiao.
\newblock Understanding the role of momentum in stochastic gradient methods.
\newblock In {\em Advances in Neural Information Processing Systems}, pages
  9633--9643, 2019.

\bibitem{GW2008}
A.~Griewank and A.~Walther.
\newblock {\em Evaluating derivatives: principles and techniques of algorithmic
  differentiation}, volume 105.
\newblock Siam, 2008.

\bibitem{HHLS2019}
M.~Hasannasab, J.~Hertrich, F.~Laus, and G.~Steidl.
\newblock Alternatives to the {EM} algorithm for {ML} estimation of location,
  scatter matrix, and degree of freedom of the student t distribution.
\newblock {\em Numerical Algorithms}, pages 1--42, 2020.

\bibitem{HHNPSS2020}
M.~Hasannasab, J.~Hertrich, S.~Neumayer, G.~Plonka, S.~Setzer, and G.~Steidl.
\newblock Parseval proximal neural networks.
\newblock {\em Journal of Fourier Analysis and Applications}, 26:59, 2020.

\bibitem{Hertrich2020}
J.~Hertrich.
\newblock {\em Superresolution via Student-$t$ Mixture Models}.
\newblock Master Thesis, TU Kaiserslautern, 2020.

\bibitem{HNS2020}
J.~Hertrich, S.~Neumayer, and G.~Steidl.
\newblock Convolutional proximal neural networks and plug-and-play algorithms.
\newblock {\em arXiv preprint arXiv:2011.02281}, 2020.

\bibitem{Higham08}
N.~J. Higham.
\newblock {\em Functions of Matrices: Theory and Computation}.
\newblock SIAM, Philadelphia, 2008.

\bibitem{HS2013}
R.~A. Horn and C.~R. Johnson.
\newblock {\em Matrix Analysis}.
\newblock Oxford University Press, 2013.

\bibitem{JZ2013}
R.~Johnson and T.~Zhang.
\newblock Accelerating stochastic gradient descent using predictive variance
  reduction.
\newblock {\em Advances in Neural Information Processing Systems}, pages
  315--323, 2013.

\bibitem{KB2014}
D.~P. Kingma and J.~Ba.
\newblock Adam: A method for stochastic optimization.
\newblock {\em ArXiv preprint arXiv:1412.6980}, 2014.

\bibitem{LLT89}
K.~L. Lange, R.~J. Little, and J.~M. Taylor.
\newblock Robust statistical modeling using the $t$ distribution.
\newblock {\em Journal of the American Statistical Association},
  84(408):881--896, 1989.

\bibitem{LS2019}
F.~Laus and G.~Steidl.
\newblock Multivariate myriad filters based on parameter estimation of
  {S}tudent-{$t$} distributions.
\newblock {\em SIAM Journal on Imaging Sciences}, 12(4):1864--1904, 2019.

\bibitem{L1963}
S.~{\L}ojasiewicz.
\newblock Une propri\'{e}t\'{e} topologique des sous-ensembles analytiques
  r\'{e}els.
\newblock In {\em Les \'{E}quations aux {D}\'{e}riv\'{e}es {P}artielles
  ({P}aris, 1962)}, pages 87--89. \'{E}ditions du Centre National de la
  Recherche Scientifique, Paris, 1963.

\bibitem{L1993}
S.~{\L}ojasiewicz.
\newblock Sur la g\'{e}om\'{e}trie semi- et sous-analytique.
\newblock {\em Universit\'{e} de Grenoble. Annales de l'Institut Fourier},
  43(5):1575--1595, 1993.

\bibitem{McLK1997}
G.~McLachlan and T.~Krishnan.
\newblock {\em The {EM} Algorithm and Extensions}.
\newblock John Wiley and Sons, Inc., 1997.

\bibitem{MVD97}
X.-L. Meng and D.~Van~Dyk.
\newblock The {EM} algorithm - an old folk-song sung to a fast new tune.
\newblock {\em Journal of the Royal Statistical Society: Series B (Statistical
  Methodology)}, 59(3):511--567, 1997.

\bibitem{N1983}
Y.~E. Nesterov.
\newblock A method for solving the convex programming problem with convergence
  rate {$O(1/k^{2})$}.
\newblock {\em Doklady Akademii Nauk SSSR}, 269(3):543--547, 1983.

\bibitem{NLST2017}
L.~M. Nguyen, J.~Liu, K.~Scheinberg, and M.~Tak{\'a}{\v{c}}.
\newblock Sarah: A novel method for machine learning problems using stochastic
  recursive gradient.
\newblock In {\em Proceedings of the 34th International Conference on Machine
  Learning-Volume 70}, pages 2613--2621, 2017.

\bibitem{NW12}
T.~M. Nguyen and Q.~J. Wu.
\newblock Robust {S}tudent's-$t$ mixture model with spatial constraints and its
  application in medical image segmentation.
\newblock {\em IEEE Transactions on Medical Imaging}, 31(1):103--116, 2012.

\bibitem{PM00}
D.~Peel and G.~J. McLachlan.
\newblock Robust mixture modelling using the $t$ distribution.
\newblock {\em Statistics and Computing}, 10(4):339--348, 2000.

\bibitem{PP08}
K.~B. Petersen and M.~S. Pedersen.
\newblock {\em The {M}atrix {C}ookbook}.
\newblock Lecture Notes, Technical University of Denmark, 2008.

\bibitem{PS2016}
T.~Pock and S.~Sabach.
\newblock Inertial proximal alternating linearized minimization (i{PALM}) for
  nonconvex and nonsmooth problems.
\newblock {\em SIAM Journal on Imaging Sciences}, 9(4):1756--1787, 2016.

\bibitem{P1964}
B.~T. Polyak.
\newblock Some methods of speeding up the convergence of iteration methods.
\newblock {\em USSR Computational Mathematics and Mathematical Physics},
  4(5):1--17, 1964.

\bibitem{Q1999}
N.~Qian.
\newblock On the momentum term in gradient descent learning algorithms.
\newblock {\em Neural networks}, 12(1):145--151, 1999.

\bibitem{RHSPS2016}
S.~J. Reddi, A.~Hefny, S.~Sra, B.~P\'oczos, and A.~Smola.
\newblock Stochastic variance reduction for nonconvex optimization.
\newblock In {\em Proc. 33rd International Conference on Machine Learning},
  2016.

\bibitem{RW98}
R.~T. Rockafellar and R.~J. Wets.
\newblock {\em Variational Analysis}, volume 317 of {\em A Series of
  Comprehensive Studies in Mathematics}.
\newblock Springer, Berlin, Heidelberg, 1998.

\bibitem{RHW1986}
D.~E. Rumelhart, G.~E. Hinton, and R.~J. Williams.
\newblock Learning representations by back-propagating errors.
\newblock {\em nature}, 323(6088):533--536, 1986.

\bibitem{SNG07}
G.~Sfikas, C.~Nikou, and N.~Galatsanos.
\newblock Robust image segmentation with mixtures of {S}tudent's
  $t$-distributions.
\newblock In {\em 2007 IEEE International Conference on Image Processing},
  volume~1, pages I -- 273--I -- 276, 2007.

\bibitem{SMDH2013}
I.~Sutskever, J.~Martens, G.~Dahl, and G.~Hinton.
\newblock On the importance of initialization and momentum in deep learning.
\newblock In {\em International conference on machine learning}, pages
  1139--1147, 2013.

\bibitem{VS14}
A.~Van Den~Oord and B.~Schrauwen.
\newblock The {S}tudent-$t$ mixture as a natural image patch prior with
  application to image compression.
\newblock {\em Journal of Machine Learning Research}, 15(1):2061--2086, 2014.

\bibitem{vanDyk1995}
D.~A. van Dyk.
\newblock {\em Construction, implementation, and theory of algorithms based on
  data augmentation and model reduction}.
\newblock PhD Thesis, The University of Chicago, 1995.

\bibitem{XY2015}
Y.~Xu and W.~Yin.
\newblock Block stochastic gradient iteration for convex and nonconvex
  optimization.
\newblock {\em SIAM Journal on Optimization}, 25(3):1686--1716, 2015.

\bibitem{YYG2018}
Z.~Yang, Z.~Yang, and G.~Gui.
\newblock A convex constraint variational method for restoring blurred images
  in the presence of alpha-stable noises.
\newblock {\em Sensors}, 18(4):1175, 2018.

\bibitem{ZZDZC14}
Z.~Zhou, J.~Zheng, Y.~Dai, Z.~Zhou, and S.~Chen.
\newblock Robust non-rigid point set registration using {S}tudent's-$t$ mixture
  model.
\newblock {\em PloS one}, 9(3):e91381, 2014.

\end{thebibliography}
\end{document}